\documentclass[reqno]{amsart}
\usepackage{paralist}
\usepackage{url}
\usepackage{mathrsfs}
\usepackage{hyperref}
\usepackage{color}
\usepackage{graphicx}
\usepackage{caption}
\usepackage{subcaption}
\usepackage{multirow}
\usepackage{mathtools}
\usepackage{tikz}
\usepackage{csvsimple}
\usepackage{booktabs}
\usepackage{amsmath}
\usepackage{amsfonts,amssymb}
\usepackage{cleveref}

\crefrangelabelformat{equation}{(#3#1#4--#5#2#6)}
\crefname{equation}{Eq.}{inequality}
\Crefname{equation}{Equation}{inequality}

\newcommand{\imag}{{\text{Im}}}
\newcommand{\real}{{\text{Re}}}

\newcommand{\bC}{{\mathbb C}}

\newcommand{\cF}{{\mathcal F}}
\newcommand{\cO}{{\mathcal O}}

\newcommand{\cI}{{\Omega}}

\newcommand{\sB}{{\mathscr B}}
\newcommand{\re}{\mathbb{R}}

\newcommand{\N}{\mathbb{N}}

\newcommand{\diag}{\mbox{diag}}

\newcommand{\eps}{\epsilon}

\newcommand{\Sig}{\Sigma}

\newcommand{\reff}[1]{(\ref{#1})}

\newcommand{\mc}[1]{\mathcal{#1}}

\newcommand{\bdes}{\begin{description}}
    \newcommand{\edes}{\end{description}}
\newcommand{\bal}{\begin{align}}
\newcommand{\eal}{\end{align}}
\newcommand{\bnum}{\begin{enumerate}}
    \newcommand{\enum}{\end{enumerate}}
\newcommand{\bit}{\begin{itemize}}
    \newcommand{\eit}{\end{itemize}}
\newcommand{\bea}{\begin{eqnarray}}
\newcommand{\eea}{\end{eqnarray}}
\newcommand{\be}{\begin{equation}}
\newcommand{\ee}{\end{equation}}
\newcommand{\baray}{\begin{array}}
    \newcommand{\earay}{\end{array}}
\newcommand{\bsry}{\begin{subarray}}
    \newcommand{\esry}{\end{subarray}}
\newcommand{\bca}{\begin{cases}}
    \newcommand{\eca}{\end{cases}}
\newcommand{\bcen}{\begin{center}}
    \newcommand{\ecen}{\end{center}}
\newcommand{\bbm}{\begin{bmatrix}}
    \newcommand{\ebm}{\end{bmatrix}}
\newcommand{\bmx}{\begin{matrix}}
    \newcommand{\emx}{\end{matrix}}
\newcommand{\bpm}{\begin{pmatrix}}
    \newcommand{\epm}{\end{pmatrix}}
\newcommand{\btab}{\begin{tabular}}
    \newcommand{\etab}{\end{tabular}}

\newtheorem{theorem}{Theorem}[section]

\newtheorem{lemma}[theorem]{Lemma}

\theoremstyle{definition}

\newtheorem{alg}[theorem]{Algorithm}

\newtheorem{remark}[theorem]{Remark}

\numberwithin{equation}{section}

\begin{document}
\title[Gaussian Mixture Models and Tensor Decompositions]
{Diagonal Gaussian Mixture Models and Higher Order Tensor Decompositions}

\author[Bingni Guo]{Bingni~Guo\textsuperscript{$\dagger$}}
\address{\textsuperscript{$\dagger$} Department of Mathematics, University of California San Diego,
9500 Gilman Drive, La Jolla, CA, USA, 92093.}
\email{\textsuperscript{$\dagger$}b8guo@ucsd.edu, njw@math.ucsd.edu}

\author[Jiawang Nie]{Jiawang~Nie\textsuperscript{$\dagger$}}

\author[Zi Yang]{Zi~Yang\textsuperscript{$\ddagger$}}
\address{\textsuperscript{$\ddagger$} Department of Mathematics and Statistics, University at Albany, State University of New York,
1400 Washington Ave, Albany, NY, USA, 12222.}
\email{\textsuperscript{$\ddagger$}zyang8@albany.edu}



\begin{abstract}
This paper studies how to recover parameters in diagonal Gaussian mixture models using tensors. High-order moments of the Gaussian mixture model are estimated from samples. They form incomplete symmetric tensors generated by hidden parameters in the model. We propose to use generating polynomials to compute incomplete symmetric tensor approximations. The obtained decomposition is utilized to recover parameters in the model. We prove that our recovered parameters are accurate when the estimated moments are accurate. Using high-order moments enables our algorithm to learn Gaussian mixtures with more components. For a given model dimension and order, we provide an upper bound of the number of components in the Gaussian mixture model that our algorithm can compute.
\end{abstract}

\maketitle

\section{Introduction}

Gaussian mixture models are widely used in statistics and machine learning \cite{karpagavalli2016review,povey2011subspace,reynolds1995speaker,zivkovic2004improved,lee2005effective,veracini2009fully} because of their simple formulation and superior fitting ability. Every smooth density can be approximated by Gaussian mixture models with enough components \cite{goodfellow2016deep}. A Gaussian mixture model is a mixture of a few Gaussian distribution components. We consider a Gaussian mixture model of dimension $d$ with $r$ components. Suppose that the $i$th component obeys the Gaussian distribution with mean $\mu_i \in \mathbb{R}^d$ and covariance matrix $\Sigma_i\in \mathbb{R}^{d\times d}$ and the probability of the $i$th component is $\omega_i>0$. Let $y\in\mathbb{R}^d$ be the random vector for the Gaussian mixture, then its density function is
\begin{equation*}
\sum_{i=1}^r\omega_i
\frac{1}{\sqrt{(2\pi)^d \det \Sigma_i} }
\exp\Big\{-\frac{1}{2}(y-\mu_i)^T \Sigma_i^{-1} (y-\mu_i)  \Big\}.
\end{equation*}
Learning the Gaussian mixture model is to recover the unknown model parameters $\omega_i,\mu_i,\Sigma_i$ for $i\in[r]$ from given independent and identically distributed (i.i.d.) samples from the Gaussian mixture. The curse of dimensionality \cite{anderson2014more} is a challenging problem when learning Gaussian mixtures with full covariance matrices since the size of full covariance matrices grows quadratically with respect to the dimension $d$. Therefore, using a diagonal matrix to approximate the full variance matrix is a common approach to fit high dimensional Gaussian mixture models \cite{magdon2010approximating}.

Throughout this paper, we focus on learning Gaussian mixture models with diagonal covariance matrices, i.e.,
\begin{equation*}
\Sigma_i=\diag(\sigma_{i1}^2,\ldots,\sigma_{id}^2),\quad i=1,\ldots,r.	
\end{equation*}
Let $M_m:=\mathbb{E}[y^{\otimes m}]$ be the $m$th order moment of the Gaussian mixture and
\begin{equation} \label{eq:tensorF}
    \cF_{m}:= \omega_1 \mu_1^{\otimes m}+\cdots  \omega_r \mu_r^{\otimes m} .
\end{equation}
The third order moment $M_3$ satisfies that
\begin{equation}\label{incomplete_tensor}
	(M_3)_{i_1i_2i_3}=(\cF_3)_{i_1i_2i_3}\quad\text{for }i_1\neq i_2\neq i_3\neq i_1.
\end{equation}
Based on the property, the work \cite{guo2021learning} performs the incomplete tensor decomposition to learn diagonal Gaussian mixture models using partially given entries of the moment tensor when $r\leq \frac{d}{2}-1$. This result uses the first and third-order moments to recover unknown model parameters. However, the third order moment $M_3$ is insufficient when $r > \frac{d}{2}-1$.

In this paper, we propose to utilize higher-order moments to learn Gaussian mixture models
with more components. The higher-order moments can be expressed
by means and covariance matrices as in the work \cite{ge2015learning}.
%
%
Let ${z}=(z_1,\cdots,z_{t})$ be a multivariate Gaussian random vector
with mean $\mu$ and covariance $\Sigma$, then
\begin{equation}   \label{moment-structure}
    \mathbb{E}[z_1\cdots z_{t}]=
    \sum_{ \substack{\lambda\in P_t \\  \lambda=\lambda_{p}\cup \lambda_s}  }\,
    \prod_{(u,v)\in \lambda_{p}}\Sigma_{u,v}\prod\limits_{c\in \lambda_{s}}\mu_c,
\end{equation}
where $P_t$ contains all distinct ways of partitioning $z_1,\cdots,z_{t}$ into two parts,
one part $\lambda_{p}$ represents $p$ pairs of $(u,v)$,
and another part $\lambda_s$ consists of $s$ singletons of $(c)$,
where $p\geq 0,\, s \geq 0$ and $2p+s=t$.
%
%

We denote the label set $\Omega_m$ for $m$th order tensor $M_m$
\begin{equation}\label{index-set}
	\Omega_m=\{(i_1,\ldots,i_m): i_1,\ldots,i_m \text{ are distinct from each other}\}.
\end{equation}
Let $z_i\sim\mathcal{N}(\mu_i,\Sigma_i)$ be the random vector for the $i$th component of the diagonal Gaussian mixture model. For $(i_1,\ldots,i_m)\in\Omega_m$,
the expression \ref{moment-structure} implies that
\begin{align}
    (M_m)_{i_1\ldots i_m}&=\sum\limits_{i=1}^r\omega_i\left(\mathbb{E}[z_i^{\otimes m}]\right)_{i_1\ldots i_m}\notag\\
    &=\sum\limits_{i=1}^r\omega_i\mathbb{E}[(z_i)_{i_1}\cdots(z_i)_{i_m}]\notag\\
    &=\sum\limits_{i=1}^r\omega_i
     \sum_{ \substack{\lambda\in P_t \\  \lambda=\lambda_{p}\cup \lambda_s}  }\,
    \prod_{(u,v)\in \lambda_{p}}(\Sigma_i)_{u,v}\prod\limits_{c\in \lambda_{s}}(\mu_i)_c \notag\\
    &=\sum\limits_{i=1}^r\omega_i (\mu_i)_{i_1}\cdots (\mu_i)_{i_m} \notag \\
    &= (\cF_m)_{i_1\ldots i_m}, \notag
\end{align}
where $P_t$ contains all distinct ways of partitioning $\{i_1,\ldots,i_m\}$ into two parts and $\lambda_p,\lambda_s$ are similarly defined as in \ref{moment-structure}. When $\lambda_p\neq \emptyset$, we have $(\Sigma_i)_{u,v}=0$  for diagonal covariance matrices. Thus, we only need to consider $\lambda_p=\emptyset$ and $\lambda_s=\{i_1,\ldots,i_m\}$. It demonstrates the above equations. We conclude that the moment tensors for diagonal Gaussian mixtures satisfy
\be\label{high-tensor-moment}
(M_m)_{i_1\ldots i_m}=(\cF_m)_{i_1\ldots i_m}
\ee
where $\cF_m=\sum\limits_{i=1}^r \omega_i\mu_i^{\otimes m}$ and $(i_1,\ldots,i_m)\in\Omega_m$.

When learning the Gaussian mixture from samples, the moment tensors $M_m$ can be well approximated by samples. Therefore, the entries $(\cF_m)_{\Omega_m}$ can be obtained from the approximation of $M_m$ by \eqref{high-tensor-moment}. This leads to the incomplete tensor decomposition problem. For an $m$th order symmetric tensor, $\cF_{i_1\ldots i_m}$ is known if $(i_1,\ldots,i_m)\in\Omega_m$, we wish to find vectors $q_1,\ldots,q_r$ such that
\begin{equation}\label{high-incomplete-decomp}
	\cF_{i_1\ldots i_m}=(q_1^{\otimes m}+\cdots+q_r^{\otimes m})_{i_1\ldots i_m}\quad\text{for all }(i_1,\ldots,i_m)\in\Omega_m.
\end{equation}
To solve the above incomplete tensor decomposition, one approach is
to complete the unknown entries first and then do the regular tensor decomposition.
We refer to \cite{FriLim18,MHWG,njwSTNN17,TanSha15,YuaZha16} for low rank tensor completion.
However, the low-rank completion is not guaranteed.

In this paper, we propose an algorithm to solve the incomplete tensor decomposition problem by generating polynomials and then utilize the decomposition to recover parameters in the Gaussian mixture. They are described in Algorithm \ref{algo:iSTD} and Algorithm \ref{algo:gaussian}. For a fixed dimension and tensor order, we provide the largest number of components, i.e., the tensor rank, our algorithm can compute in Theorem \ref{largest-r}.


\subsection*{Related Work} \, The Gaussian mixture model is a classic statistical model with broad applications in applied statistics and machine learning. Many techniques have been proposed to estimate unknown parameters for Gaussian mixture models given samples.

Expectation-maximization(EM) algorithm is one popular method that begins with some initial estimates of the mixture of Gaussians and then iteratively updates the approximation to improve its likelihood\cite{dempster1977maximum}. It is well known that it often gets stuck in local optima or fails to converge\cite{redner1984mixture}. Dasgupta \cite{dasgupta1999learning} proposed an algorithm to recover parameters in randomly low-dimensional projected subspace and required certain separation conditions for the mean vectors of the model. We refer to \cite{dasgupta2013two, sanjeev2001learning, vempala2004spectral, achlioptas2005spectral} more works on the EM algorithm.

The method of moments is another approach, introduced by Pearson~\cite{pearson1894contributions}.
Moitra and Valiant~\cite{moitra2010settling} proposed a learning algorithm
to search for the best parameters over a discretized set to match the empirical moments. Kalai, et al.~\cite{kalai2010efficiently}
proposed a randomized algorithm for a mixture of two high dimensional Gaussians with arbitrary covariance matrices.
The algorithm reduces the model to a univariate problem and solves it using moments up to order six. Hsu and Kakade~\cite{hsu2013learning} and Anandkumar et al.~\cite{anandkumar2017analyzing} used moments up to order three to learn Gaussian mixtures of spherical covariances $(\Sigma_i=\sigma_i^2 I)$. Ge et al.~\cite{ge2015learning} provided a learning algorithm for Gaussian mixtures of general covariances using moments up to order six. 
This algorithm requires $d\geq\Omega(r^2)$. Algebraic methods using moments and polynomials can be applied to fit diagonal or spherical Gaussian mixtures, as in Guo et at.~\cite{guo2021learning}, Khouja et al.~\cite{khouja2022tensor}, Lindberg et al.~\cite{lindberg2021estimating}, Belkin and Sinha \cite{belkin2010polynomial}, and Bhaskara et al.~\cite{bhaskara2014smoothed}.
Recently, Pereira et al. \cite{PKK22}, Sherman and Kolda \cite{sherman2020estimating}
proposed moment-matching approaches based on numerical optimization. 
These approaches do not need the explicit formulation of large symmetric tensors, for reducing computation and storage costs.



\subsection*{Contributions}  This paper proposes a novel algorithm to learn diagonal Gaussian mixture models. The proposed algorithm is based on high-order moment tensors, incomplete tensor decompositions, and generating polynomials. Let $y_1,\ldots,y_N$ be i.i.d. samples from the diagonal Gaussian mixture. Learning the Gaussian mixture is to estimate the model parameters $\{\omega_i,\mu_i,\Sigma_i\}_{i=1}^r$, where each $\Sigma_i$ is diagonal.

Let $M_{m}$ be the $m$th order moment of the Gaussian mixture respectively. For the tensors in \eqref{eq:tensorF}, the entries $(\cF_{m})_{i_1\ldots i_m}$ are known from $M_m$ for distinct labels $i_1,\ldots,i_m$ by \eqref{high-tensor-moment}. We use generating polynomials to compute the decomposition of $\cF_{m}$. We first construct a set of linear systems that only use the known entries in $\cF_m$ and then solve the linear systems to get a set of special generating polynomials. The common zeros of the generating polynomials can be obtained from corresponding eigenvalue decompositions. Under some genericity conditions, we can recover the decomposition of $\cF_m$ using the common zeros. Finally, the unknown parameters $\{\omega_i,\mu_i,\Sigma_i\}_{i=1}^r$ can be obtained by solving linear systems using the tensor decomposition and lower order moments. In real applications, the moment $M_m$ is estimated from samples, which are not accurately given. In such a case, our algorithm can still find a good approximation of $\cF_m$ and use the approximation to estimate the unknown parameters. We prove that the estimated parameters are accurate provided that the estimated moment tensors are accurate.

The previous work \cite{guo2021learning} only uses the first and third order moments to learn the Gaussian mixture model. Thus, the number of components must satisfy $r\le \frac{d}{2}-1$. The proposed algorithms in this work use high order moments and hence can learn Gaussian mixture models with more components. For a given highest order $m$, we provide the largest number of components our algorithm can compute in Theorem \ref{largest-r}. Conversely, Theorem \ref{largest-r} can be used to find the smallest order $m$ needed for the given number of components.

The paper is organized as follows. In Section~\ref{sc:pre}, we review symmetric tensor decompositions and generating polynomials. In Section~\ref{sc:tdc}, we develop a novel algorithm that uses generating polynomials to solve the incomplete tensor decomposition problem arising from the Gaussian mixture problem. Section~\ref{sc:approximation} discusses how to find the tensor approximation when errors exist and provides the error analysis. Section~\ref{sc:gmm} provides the algorithm to learn Gaussian mixture models using incomplete symmetric tensor decompositions. Section~\ref{sc:exp} presents the numerical experiments for the proposed algorithms.

\section{Preliminary}
\label{sc:pre}

\subsection*{Notation}
Denote $\mathbb{N}$, $\mathbb{C}$ and $\mathbb{R}$ the set of nonnegative integers, complex and real numbers respectively. Denote the cardinality of a set $L$ as $|L|$. Denote by $e_i$ the $i$th standard unit basis vector, i.e., the $i$the entry of $e_i$ is one
and all others are zeros. For a complex number $c$, $\sqrt[n]{c}$ or $c^{1/n}$ denotes the principal $n$th root of $c$. For a complex vector $v$, $\text{Re}(v),\,\text{Im}(v)$ denotes the real part and imaginary part of $v$ respectively. A property is said to be generic if it is true in the whole space except a subset of zero Lebesgue measure.
The $\|\cdot\|$ denotes the Euclidean norm of a vector or the Frobenius norm of a matrix.
For a vector or matrix, the superscript $^T$ denotes the transpose
and $^H$ denotes the conjugate transpose. For $i,j \in \N$, $[i]$ denotes the set $\{1,2,\ldots,i\}$ and $[i,j]$ denotes the set $\{i,i+1,\ldots,j\}$ if $i \le j$.
For a vector $v$, $v_{i_1:i_2}$ denotes the vector $(v_{i_1},v_{i_1+1},\ldots,v_{i_2})$. For a matrix $A$, denote by $A_{[i_1:i_2, j_1:j_2]}$ the submatrix of $A$ whose row labels are $i_1,i_1+1,\ldots,i_2$ and whose column labels are $j_1,j_1+1,\ldots,j_2$.
For a tensor $\cF$, its subtensor is similarly defined.
%

%
%
%

Let $\textrm{S}^m(\bC^{d})$ (resp., $\textrm{S}^m(\re^{d})$) denote the space of
$m$th order symmetric tensors over the vector space $\bC^{d}$ (resp., $\re^{d}$).
For convenience of notation, the labels for tensors start with $0$.
A symmetric tensor $\mathcal{A} \in \textrm{S}^m(\bC^{n+1})$ is labelled as
\[
\mathcal{A} \, = \, (\mathcal{A}_{i_1...i_m} )_{ 0\le i_1, \ldots, i_m \le n} ,
\]
where the entry $\mathcal{A}_{i_1\ldots i_m}$ is invariant for all permutations of
$(i_1,\ldots,i_m)$. The Hilbert-Schmidt norm $\|\mathcal{A}\|$ is defined as
\begin{equation} \label{tensor:norm:HS}
\|\mathcal{A}\| \,:= \,
\Big(\sum\limits_{0\leq i_1,{\ldots},i_m\leq n}|
\mathcal{A}_{i_1{\ldots} i_m}|^2\Big)^{1/2}.
\end{equation}
The norm of a subtensor $\| \mc{A}_{\Omega} \|$ is similarly defined.
For a vector $\mu:=(\mu_0,\ldots,\mu_n)\in \bC^{n+1}$, the tensor power
$\mu^{\otimes m}:=\mu\otimes \cdots \otimes \mu$, where $\mu$ is repeated $m$ times,
is defined such that
\[
(\mu^{\otimes m})_{i_1\ldots i_m} \, = \, \mu_{i_1}
    \times \cdots \times \mu_{i_m} .
\]
For a symmetric tensor $\mathcal{F}$, its symmetric rank is
\begin{equation*}
\text{rank}_{\textrm{S}}(\mathcal{F})\coloneqq
\text{min}{\Big \{r \, \mid \, \mathcal{F}=\sum\limits_{i=1}^r \mu_i^{\otimes m}\Big\}}.
\end{equation*}
There are other types of tensor ranks \cite{Lim13}.
In this paper, we only deal with symmetric tensors and symmetric ranks.
We refer to \cite{CLQY20,DeSLim08,Fri16,HLim13,Lim13}
for general work about tensors and their ranks.
For convenience, if $r=\text{rank}_{\textrm{S}}(\mathcal{F})$,
we call $\mathcal{F}$ a rank-$r$ tensor and
$\mathcal{F}=\sum\limits_{i=1}^r \mu_i^{\otimes m}$ is called a rank decomposition.

For a power $\alpha\coloneqq(\alpha_1,\alpha_2,\cdots,\alpha_{n})\in\mathbb{N}^{n}$
and $x\coloneqq(x_1,x_2,\cdots,x_{n})$, denote
\[
|\alpha|\coloneqq\alpha_1+\alpha_2+\cdots+\alpha_{n},\quad
x^{\alpha}\coloneqq x_1^{\alpha_1}x_2^{\alpha_2}\cdots x_{n}^{\alpha_{n}}, \quad
x_0 :=1.
\]
The monomial power set of degree $m$ is denoted as
\[
\mathbb{N}^n_m \coloneqq \{\alpha=(\alpha_1,\alpha_2,\cdots,\alpha_{n})
\in\mathbb{N}^n:|\alpha|\leq m\}.
\]
The symmetric tensor $\cF\in \textrm{S}^m(\mathbb{C}^{n+1})$ 
can be labelled by the monomial power set $\mathbb{N}^n_m$, i.e.,
\[
    \cF_\alpha=\cF_{x^\alpha}=\cF_{i_1\ldots i_m}
\]
where $x^\alpha=x_0^{m-|\alpha|}x^\alpha=x_{i_1}\ldots x_{i_m}$. 

For a finite set $\mathscr{B} \subset \bC[x]$ of monomials and a vector $v\in \bC^n$, we denote the vector of monomials in $\mathscr{B}$ evaluated on $v$ as
\[
    [v]_\mathscr{B}:=(f(v))_{f\in \mathscr{B}}.
\]
%
%

Let $\mathbb{C}[x]_m$ be the space of all polynomials in $x$ with complex coefficients
and degrees no more than $m$. For a polynomial
$p\in\mathbb{C}[x]_m$ and a symmetric tensor $\mathcal{F}\in \textrm{S}^m(\mathbb{C}^{n+1})$,
we define the bilinear product (note that $x_0 = 1$)
\begin{equation}
    \langle p,\mathcal{F}\rangle=\sum\limits_{\alpha \in \N_m^n} p_\alpha \cF_\alpha
    \quad\text{for}\quad p \, =\sum\limits_{\alpha \in \N_m^n} p_\alpha x^\alpha,
\end{equation}
where $p_\alpha$'s are coefficients of $p$.
A polynomial $g\in\mathbb{C}[x]_m$ is called a {\it generating polynomial}
for a symmetric tensor $\mathcal{F} \in \textrm{S}^m(\bC^{n+1})$ if
\begin{equation}\label{gene_poly}
    \langle g\cdot x^{\beta},\ \mathcal{F}\rangle=0\quad\forall\beta\in\mathbb{N}_{m-\text{deg}(g)}^n ,
\end{equation}
where $\text{deg}(g)$ denotes the degree of $g$ in $x$.
We refer to \cite{nie2017generating} for more details of generating polynomials.
The generating polynomials are powerful tools to compute low-rank tensor decompositions and approximations \cite{guo2021learning,nie2017generating,nie2017low}.
More work about tensor optimization can be found in
\cite{DNY20,NieSOSbd,Tight19,MPO23,NieZhang18,NieYang20,NieYe19}.

\section{Incomplete Symmetric Tensor Decompositions}\label{sc:tdc}
In this section, we discuss how to solve the incomplete symmetric tensor decomposition problem arising from learning Gaussian mixtures with parameters $\{\omega_i,\mu_i,\Sigma_i\}_{i=1}^r$. Let $\cF_m\in S^{m}(\mathbb{C}^{d})$ be the symmetric tensor. In the following, we discuss how to obtain the decomposition of $\cF_{m}$ given entries $(\cF_m)_{\Omega_m}$.

For convenience, we denote $n\coloneqq d-1$. Suppose that $\cF_m$ has the decomposition
\begin{equation} \label{eq:cFm}
    \cF_m = \omega_1 \mu_1^{\otimes m}+\cdots  \omega_r \mu_r^{\otimes m},
\end{equation}
where $\mu_i=((\mu_i)_0,(\mu_i)_1,\ldots,(\mu_i)_n)\in \bC^{n+1} $.
When the leading entry of each $\mu_i$ is nonzero, we can write the decomposition \eqref{eq:cFm} as
\begin{equation}\label{F-decomp}
        \cF_m=\lambda_1\begin{bmatrix}
        1\\u_1\end{bmatrix}^{\otimes m}+\cdots+\lambda_r\begin{bmatrix}
        1\\u_r\end{bmatrix}^{\otimes m},
\end{equation}
where $\lambda_i=\omega_i ((\mu_1)_0)^m$, and $u_i=((u_i)_1,\ldots,(u_i)_n)=(\mu_i)_{1:n}/(\mu_i)_0\in \bC^n$.

Let $1\le p \le m-2$ and $p\le k \le n-m-p$ 
be numbers such that 
\[
\binom{k}{p}\ge r  \quad \text{and} \quad  \binom{n-k-1}{m-p-1}\ge r .
\]
Define the set
\begin{equation}\label{B0}
\mathscr{B}_0\subseteq\{x_{i_1}\cdots x_{i_{p}}:1\leq i_1<\cdots<i_{p}\leq k\}
\end{equation}
such that $\mathscr{B}_0$ consists of the first $r$ monomials in the graded lexicographic order. 
Correspondingly, the set $\mathscr{B}_1$ is defined as
\begin{equation}\label{B1}
\mathscr{B}_1\coloneqq\{x_{j_1}\cdots x_{j_{p+1}}:1\leq j_1<\cdots<j_{p}\leq k<j_{p+1}\leq n\}.\end{equation}
For convenience, we say $\alpha \in \mathbb{N}^n$ is in $\mathscr{B}_0$ (resp. $\mathscr{B}_1$) 
if $x^\alpha \in $ $\mathscr{B}_0$ (resp. $\mathscr{B}_1$).
Let $\alpha=e_{j_1}+\cdots+e_{j_p}+e_{j_{p+1}}\in\mathscr{B}_1$ and 
$G\in\mathbb{C}^{\mathscr{B}_0\times\mathscr{B}_1}$ 
be a matrix labelled by monomials in $\mathscr{B}_0$ and $\mathscr{B}_1$. We consider the polynomial
\[
\varphi_{j_1\cdots j_{p}j_{p+1}}[G](x) \, \coloneqq  \,
   \sum\limits_{(i_1,\ldots,i_p)\in\mathscr{B}_0}G(\sum_{t=1}^pe_{i_t},\,\sum_{t=1}^{p+1}e_{j_t})x_{i_1}\cdots x_{i_{p}}
   -x_{j_1}\cdots x_{j_{p}}x_{j_{p+1}}.
\]
Recall that $\varphi_{j_1\cdots j_{p+1}}[G](x)$ 
is a generating polynomial for $\cF_m$ if it satisfies \reff{gene_poly}, i.e.
\begin{equation*}
	\langle \varphi_{j_1\cdots j_{p}j_{p+1}}[G](x)\cdot x^{\beta}, \ \mathcal{F}_m\rangle=0\quad\forall\beta\in\mathbb{N}_{m-p-1}^n.
\end{equation*}
The matrix $G$ is called a generating matrix if $\varphi_{j_1\cdots j_{p+1}}[G](x)$ is a generating polynomial.
If the matrix $G$ is a generating matrix of $\cF_m$, it should satisfy the equations
\begin{equation}\label{generating_mat}
\langle x_{s_1}\cdots x_{s_{m-p-1}}\varphi_{j_1\cdots j_{p+1}}[G](x),\,\cF_m\rangle = 0
\end{equation}
for each $\alpha=e_{j_1}+\cdots+e_{j_p}+e_{j_{p+1}}\in\mathscr{B}_1$ and each tuple $(s_1,\ldots,s_{m-p-1})\in\mathcal{O}_{\alpha}$, where
\[
\mathcal{O}_{\alpha}\coloneqq \left\{(s_1,\ldots,s_{m-p-1}):
\begin{array}{l}
     k+1\leq s_1<\ldots<s_{m-p-1}\leq n,  \\
      s_1\ne j_{p+1},\ldots, s_{m-p-1}\ne j_{p+1}
\end{array}\right\}.
\]
Define the matrix $A[\alpha,\cF_m]$ and the vector $b[\alpha,\cF_m]$ be such that
\begin{equation}\label{eq: Aalpha}
\left\{
\begin{aligned}
A[\alpha,\cF_m]_{\gamma,\beta}&\coloneqq (\cF_m)_{\beta+\gamma},\quad\forall(\gamma,\beta)\in\mathcal{O}_{\alpha}\times\mathscr{B}_0\\
b[\alpha,\cF_m]_{\gamma}&\coloneqq(\cF_m)_{\alpha+\gamma},\quad\forall\gamma\in\mathcal{O}_{\alpha}.
\end{aligned} \right.
\end{equation}
The dimension of $A[\alpha,\cF_m]$ is $\binom{n-k-1}{m-p-1}\times r$ and the equations in \reff{generating_mat} can be equivalently written as
\begin{equation}\label{solve-G}
	A[\alpha,\cF_m]\cdot G(:,\alpha)=b[\alpha,\cF_m].
\end{equation}

Lemma \ref{lemma: lid A} proves that the matrix $A[\alpha,\cF_m]$  in \eqref{eq: Aalpha} has full column rank under some genericity conditions.
\begin{lemma}\label{lemma: lid A}
    Suppose that $\binom{k}{p}\ge r$ and $\binom{n-k-1}{m-p-1}\ge r$. Let $\cF_m$ be the tensor with the decomposition \eqref{F-decomp}.
    If vectors $\{[u_i]_{\sB_0}\}_{i=1}^r$ and $\{[u_i]_{\cO_\alpha}\}_{i=1}^r$ are both linearly independent, then the matrix $A[\alpha,\cF_m]$ as in \eqref{eq: Aalpha} has full column rank.
\end{lemma}
\begin{proof}
 The matrix $A[\alpha,\cF_m]$ can be written as
\[
    A[\alpha,\cF_m] = \sum_{i=1}^r \lambda_i[u_i]_{\cO_\alpha} [u_i]_{\sB_0}^T.
\]
Therefore,  $A[\alpha, \cF_m]$ has full column rank.
\end{proof}

\begin{remark}
A successful construction of $\mathscr{B}_0$ requires that $\binom{k}{p}\ge r$. 
The vectors $\{[u_i]_{\sB_0}\}_{i=1}^r$ and $\{[u_i]_{\cO_\alpha}\}_{i=1}^r$
have dimensions $r$ and $\binom{n-k-1}{m-p-1}$ respectively. Thus, when
\[
\binom{k}{p}\ge r \quad \text{and}  \quad  \binom{n-k-1}{m-p-1}\ge r,
\]
the vectors $\{[u_i]_{\sB_0}\}_{i=1}^r$ and $\{[u_i]_{\cO_\alpha}\}_{i=1}^r$
are both linearly independent for generic vectors $u_1,\ldots,u_r$
in real or complex field.
\end{remark}
Under the condition of Lemma \ref{lemma: lid A}, we can prove there exists a unique generating matrix $G$ for $\cF_m$.

\begin{theorem} \label{thm: G}
    Let $\cF_m$ be the tensor in \eqref{F-decomp}. Suppose that conditions of Lemma \ref{lemma: lid A} hold, then there exists a unique generating matrix $G$ for the tensor $\cF_m$.
\end{theorem}
\begin{proof}
    We first prove the existence of $G$. For $1\le i \le r$, $k+1 \le j \le n$, we denote
    \[
        d_j = ((u_1)_{j},\ldots,(u_r)_j).
    \]
    Under the assumption of Lemma \ref{lemma: lid A}, we can define
    \[
        N_j =  ([u_1]_{\sB_0},\ldots,[u_r]_{\sB_0}) \diag(d_j) ([u_1]_{\sB_0},\ldots,[u_r]_{\sB_0})^{-1}.
    \]  The matrix $G$ is constructed as 
    \[
        G(\beta, \nu+e_j) = N_j(G)_{\nu,\beta} 
    \]
    for $j=k+1,\ldots, n$ and $\nu,\beta\in\mathscr{B}_0$. For every $\alpha=\nu+e_j \in \mathscr{B}_1$, it holds that 
    \[
        \sum_{\beta \in \sB_0} G(\beta,\alpha) u_i^{\beta} - u_i^{\alpha} = N_j(G)_{\nu,:}[u_i]_{\sB_0}-(u_i)_j u_i^{\nu}=0
    \]
    Thus, for every $\gamma \in \N^{n}_{m-p-1}$, it holds that
    \[
        \langle x^{\gamma} \varphi_\alpha[G](x), \cF_m\rangle = \sum_{i=1}^r \lambda_i u_i^\gamma\sum_{\theta\in \sB_0} (G(\theta,\alpha) u_i^{\theta} - u_i^{\alpha}) = 0.
    \]
    It proves that the matrix $G$ is a generating matrix for $\cF_m$.

    Next, we show the uniqueness. The matrix $A[\alpha,\cF_m]$ has full column rank by Lemma \ref{lemma: lid A}, so the generating matrix $G$ is uniquely determined by linear systems in \eqref{solve-G}. It proves the uniqueness of $G$.
\end{proof}

By Theorem \ref{thm: G} and Lemma \ref{lemma: lid A}, the generating matrix $G$ can be uniquely determined by solving the linear system \eqref{solve-G}.
Let $N_{k+1}(G),\ldots,N_n(G)\in\mathbb{C}^{r\times r}$ be the matrices given as $(\nu,\beta\in\mathscr{B}_0)$:
\begin{equation}\label{eq:N}
N_l(G)_{\nu,\beta} =  G(\beta,\ \nu+e_l)\quad \text{for }l=k+1,\ldots, n .
\end{equation}
Then we have
\begin{equation*}
N_l(G)[v_i]_{\mathscr{B}_0}=(w_i)_{l-k}	[v_i]_{\mathscr{B}_0}\quad \text{for }l=k+1,\ldots, n .
\end{equation*}
for the vectors ($i=1,\ldots,r$)
\begin{equation*}
v_i\coloneqq ((v_i)_1,\ldots,(v_i)_k) =(u_i)_{1:k},\, w_i\coloneqq((w_i)_1,\ldots,(w_i)_{n-k}) =(u_i)_{k+1:n}.	
\end{equation*}
We select a generic vector $\xi\coloneqq(\xi_{k+1},\ldots,\xi_n)$ and let
\begin{equation}\label{N_xi}
N(\xi) \, \coloneqq  \, \xi_{k+1}N_{k+1}+\cdots+\xi_{n}N_{n}.
\end{equation}
Let $\tilde{v}_1,\ldots,\tilde{v}_r$ be unit length eigenvectors of $N(\xi)$, which are also common eigenvectors of $N_{k+1}(G),\ldots,N_{n}(G)$. For each $i=1,\ldots,r$, let $\tilde{w}_i$ be the vector such that its $j$th entry
$(\tilde{w}_i)_j$ is the eigenvalue of $N_{k+j}(G)$,
associated to the eigenvector $\tilde{v}_i$. Equivalently,
\begin{equation}\label{w_i}
    \tilde{w}_i\coloneqq(\tilde{v}_i^H N_{k+1}(G)\tilde{v}_i,\cdots,
    \tilde{v}_i^H N_n(G)\tilde{v}_i)\quad i=1,\ldots,r.
\end{equation}
Up to a permutation of $(\tilde{v}_1,\ldots, \tilde{v}_r)$, we have
\begin{equation*}
w_i=\tilde{w}_i.	
\end{equation*}




We denote the sets
\begin{equation}\label{label-set-J}
\boxed{\begin{array}{l}
J_1 \coloneqq \{x_{i_1}\cdots x_{i_p}:1\leq i_1<\cdots<i_p\leq k\},\\
J_1^{-j}  \coloneqq J_1 \cap \{x_{i_1}\cdots x_{i_p}: i_1,\ldots,i_p \neq j\},\\
J_2 \coloneqq \{(x_{i_1}\cdots x_{i_{m-p-1}}:k+1\leq i_1<\cdots<i_{m-p-1}\leq n\},\\
J_3 \coloneqq \{x_{i_1-k}\cdots x_{i_{m-p-1}-k}:(i_1,\ldots,i_{m-p-1})\in J_2\}.\\
\end{array}}
\end{equation}
The tensors $\lambda_1v_1^{\otimes p},\ldots,\lambda_rv_r^{\otimes p}$ satisfy the linear equation
\begin{equation*}
\sum_{i=1}^r \lambda_i v_i^{\otimes p} \otimes \tilde{w}_i^{\otimes (m-p-1)}=(\cF_m)_{0,[1:k]^{p},[k+1:n]^{(m-p-1)}}.
\end{equation*}
Thus, $\lambda_1[v_1]_{J_1},\ldots,\lambda_r[v_r]_{J_1}$ can be obtained by the linear equation
\begin{equation}\label{coef-1}
    \min\limits_{(\gamma_1,\ldots,\gamma_r)} \left\|(\cF_{m})_{J_1\cdot J_2}-
     \sum\limits_{i=1}^r\gamma_i \otimes [\tilde{w}_i]_{J_{3}} \right\|^2.
\end{equation}
We denote the minimizer of \reff{coef-1} by $(\tilde{\gamma}_1,\ldots,\tilde{\gamma}_r)$.

The vectors $v_1,\ldots,v_r$ satisfy the linear equation
\begin{equation*}
\sum_{i=1}^r v_i\otimes\lambda_iv_i^{\otimes p}\otimes\tilde{w}_i^{\otimes (m-p-1)}=(\cF_m)_{[1:k]^{p+1}\times[k+1:n]^{(m-p-1)}}.
\end{equation*}
For each $j\in[k]$, we solve the linear least square problem
\begin{equation}\label{coef-2}
    \min\limits_{(v_1,\ldots,v_r)} \left\|(\cF_{m})_{x_j\cdot J_1^{-j}\cdot J_2}-\sum\limits_{i=1}^r(v_i)_j\cdot\tilde{\gamma}_i\otimes[\tilde{w}_i]_{J_{3}} \right\|^2.
\end{equation}
We denote the minimizer of \reff{coef-2} as $(\tilde{v}_1,\ldots,\tilde{v}_r)$.



The scalars $\lambda_1, \ldots,\lambda_r$ in \eqref{F-decomp} satisfy the linear equation
\begin{equation}
    \lambda_1\begin{bmatrix}
        1\\\tilde{u}_1
    \end{bmatrix}^{\otimes m}+\cdots+\lambda_r\begin{bmatrix}
        1\\\tilde{u}_r
    \end{bmatrix}^{\otimes m}=\cF_m,
\end{equation}
where $\tilde{u}_i=(\tilde{v_i},\tilde{w}_i)$ for $i=1,\ldots,r$. They can be solved by the following linear least square problem
\begin{equation}\label{coef-lambda}
    \min\limits_{(\lambda_1,\ldots,\lambda_r)} \left\|(\cF)_{\Omega_m}-\sum\limits_{i=1}^r\lambda_i\cdot \left(\begin{bmatrix}
        1\\\tilde{u}_i
    \end{bmatrix}^{\otimes m} \right)_{\Omega_m}\right\|^2.
\end{equation}
Let $(\tilde{\lambda}_1,\ldots,\tilde{\lambda}_r)$ be the minimizer of \reff{coef-lambda}.

Concluding everything above, we obtain the decomposition of $\cF_m$
\begin{equation*}
    \cF_m=q_1^{\otimes m}+\cdots+q_r^{\otimes m},
\end{equation*}
where $q_i\coloneqq(\tilde{\lambda})^{1/m}(1,\tilde{v}_i,\tilde{w}_i)$,
for $i=1,\ldots,r$. All steps to obtain the decomposition are summarized in Algorithm \ref{algo:iSTD}




\begin{alg}\label{algo:iSTD}
(Incomplete symmetric tensor decompositions.)
\begin{itemize}

\item [Input:]
Rank $r$, dimension $d$, constant $p$ and subtensor $(\cF_m)_{\Omega_m}$ in \eqref{F-decomp}.

\item [Step~1.] Determine the matrix $G$ by solving \reff{solve-G}
for each $\alpha=e_{j_1}+\cdots+e_{j_{p+1}}\in\mathscr{B}_1$.

\item [Step~2.]  Let $N(\xi)$ be the matrix as in \reff{N_xi},
for a randomly selected vector $\xi$. Compute the vectors $\tilde{w}_i$ as in \reff{w_i}.

\item [Step~3.] Solve the linear least squares \reff{coef-1}, \reff{coef-2} and \reff{coef-lambda}
to get the scalars $\tilde{\lambda}_i$ and vectors $\tilde{v}_i$.



\item [Output:] The tensor decomposition
$\cF_m = q_1^{\otimes m}+\cdots+q_r^{\otimes m}$, for $q_i=(\tilde{\lambda})^{1/m}(1,\tilde{v}_i,\tilde{w}_i)$.

\end{itemize}
\end{alg}

\begin{theorem} \label{thm:TD}
    Let $\cF_m$ be the tensor in \eqref{F-decomp}. If $\cF_m$ satisfies conditions of Lemma \ref{lemma: lid A} and the matrix $N(\xi)$ in \eqref{N_xi} has distinct eigenvalues, then Algorithm~\ref{algo:iSTD} finds the unique rank-$r$ decomposition of $\cF$.
\end{theorem}
\begin{proof}
    Under the assumptions of Lemma \ref{lemma: lid A}, the tensor $\cF_m$ has a unique generating matrix by Theorem \ref{thm: G} and the generating matrix $G$ is uniquely determined by solving \eqref{solve-G}. The matrix $N(\xi)$ in \eqref{N_xi} has distinct eigenvalues, so the vectors $\tilde{w}_i$ are determined by \eqref{w_i}. Lemma \ref{lemma: lid A} assumes $\{[u_i]_{\cO_\alpha}\}_{i=1}^r$ are linearly independent, it implies that $\{[\tilde{w}_i]_{J_{3}}\}_{i=1}^r$ are also linearly independent. Thus, the systems \eqref{coef-1} and \eqref{coef-2} both have unique solutions. By the uniqueness of every step in the Algorithm \ref{algo:iSTD}, we conclude that Algorithm \ref{algo:iSTD} finds the unique rank-$r$ decomposition of $\cF_m$.
\end{proof}

Algorithm \ref{algo:iSTD} requires the tensor $\cF_m$ to satisfy the condition of Lemma \ref{lemma: lid A}. Thus, the rank $r$ should satisfy
\[
    r \le \min \left \{ \binom{k}{p}, \binom{n-k-1}{m-p-1}  \right \}.
\]
In the following, we will find the largest rank that Algorithm \ref{algo:iSTD} can compute for the given order $m$.

\begin{lemma} \label{lemma:max r proof}
    If $n\ge \max \{2m-1,\frac{m^2}{4}-1\}$, then
    \begin{eqnarray} \label{eq:max r}
    &&\max \left( \binom{k^*}{p^*},  \binom{n-k^*-2}{m-p^*-1}  \right)\\
        &&=\max\limits_{p\in \mathbb{N} \cap [1,m-2] } \max\limits_{k\in \mathbb{N} \cap [p,n-m+p]}\left(\min\left(\binom{k}{p}, \binom{n-k-1}{m-p-1}\right)\right), \nonumber
    \end{eqnarray}
        where $p^*=\lfloor \frac{m-1}{2} \rfloor$ and $k^*$ is largest $k$ such that
    $\binom{k}{p^*}\le\binom{n-k-1}{m-p^*-1}$.
\end{lemma}

\begin{proof}
For a fixed $p\in [1,\frac{m-1}{2}]\cap \mathbb{N}$, it holds that $\binom{k}{p}$ is increasing in $k$ and $\binom{n-k-1}{m-p-1}$ is decreasing in $k$. For the fixed $p$, let $k_p$ be the largest $k$ such that
\[
    \binom{k}{p} \le \binom{n-k-1}{m-p-1}.
\]
It holds that
\[
   r_p:= \max\limits_{k\in \mathbb{N} \cap [p,n-m+p]}\left(\min\left(\binom{k}{p}, \binom{n-k-1}{m-p-1}\right)\right) = \max \left( \binom{k_p}{p},  \binom{n-k_p-2}{m-p-1}  \right).
\]
For $p\in(\frac{m-1}{2},m-2]\cap \mathbb{N}$ and $k\in [p,n-m+p]$, let $p'=m-p-1$ and $k'=n-k-1$. We can verify that $p'\in [1,\frac{m-1}{2}]$, $k'\in [p',n-m+p']$, and
\[
    \min\left(\binom{k}{p}, \binom{n-k-1}{m-p-1}\right) = \min\left(\binom{k'}{p'}, \binom{n-k'-1}{m-p'-1}\right).
\]
Therefore, it holds that $\max_{p \in \mathbb{N} \cap [1,m-2]}r_p=\max_{p \in \mathbb{N} \cap [1,p^* ]}r_p$. Next, we will prove $\max_p r_p = r_{p*}$ by showing $r_p\ge r_{p-1}$ for $ p \in \mathbb{N} \cap [2,p^* ]$.

When $p\le \frac{m-1}{2}$ and $n\ge 2m-1$, we have $p\le m-p-1\le n-1-\lfloor \frac{n-1}{2}\rfloor -p$. Hence,
\[
    \binom{\lfloor \frac{n-1}{2} \rfloor}{p} \le \binom{ n-1-\lfloor\frac{n-1}{2} \rfloor}{p}= \binom{ n-1-\lfloor\frac{n-1}{2} \rfloor}{n-1-\lfloor\frac{n-1}{2}\rfloor-p}\le \binom{ n-1-\lfloor\frac{n-1}{2} \rfloor}{m-p-1}.
\]
The above equation implies $k_p\ge \lfloor \frac{n-1}{2} \rfloor \ge \frac{n-2}{2}$.

If $k_{p-1}< k_p$, then it holds that
\[
    r_{p-1}\le\binom{k_{p-1}+1}{p-1}\le\binom{k_p}{p-1}= \binom{k_p}{p} \frac{p}{k_p-p+1}\le \binom{k_p}{p} \frac{m-1}{n-m} \le \binom{k_p}{p} \le r_p.
\]
In the following proof, we show $r_{p-1}\le r_p$ if $k_{p-1}\ge k_p$.

\textbf{Case 1:} $\binom{k_p}{p}>\binom{n-k_p-2}{m-p-1}$. In this case, $r_p=\binom{k_p}{p}$. It holds that
\begin{eqnarray*}
    \binom{k_p'-C}{p'+1}&=& \binom{k_p'}{p'} \frac{k_p'-p'}{p'+1}\Pi_{i=1}^C \frac{k_p'-i-p'}{k_p'-i+1},
\end{eqnarray*}
\begin{eqnarray*}
    \binom{k_p+C}{p-1}&=&\binom{k_p}{p}\frac{p}{k_p-p+1}\Pi_{i=1}^C \frac{k_p+i}{k_p-p+1+i},
\end{eqnarray*}
for $C\ge 0, p'=m-p-1,k_p'=n-k_p-1$. By direct computation, we have
\[
    \frac{k_p'-i-p'}{k_p'-i+1} \frac{k_p+i}{k_p-p+1+i} \le 1 \Leftrightarrow k_pm-np+n-k_p+im-i\ge 0.
\]
The inequalities $n-m+p\ge k_p\ge \frac{n-2}{2},p\le \frac{m-1}{2},n\ge 2m-1,i\ge 0$ imply
\[
    k_pm-np +n-k_p+im-i \ge n-m + i(m-1)+1 \ge 0.
\]
It proves 
\begin{equation} \label{eq:ineq}
\frac{k_p'-i-p'}{k_p'-i+1} \frac{k_p+i}{k_p-p+1+i} \le 1,\, \text{for $i\ge 0$. }
\end{equation}

If $p=\frac{m-1}{2}$, then $p=m-p-1=p'$ and $k_p=\lfloor \frac{n-1}{2} \rfloor$. If $n$ is even, then $k_p=\frac{n-2}{2}$ and $\binom{k_p}{p}=\binom{n-k_p-2}{m-p-1}$, which does not satisfy the assumption of Case 1. If $n$ is odd, then $k_p=\frac{n-1}{2}$ and $k_p=k_p'$. Thus, we have $\binom{k_p'}{p'}=\binom{k_p}{p}$ and $\frac{k_p'-p'}{p'+1}\frac{p}{k_p-p+1}=\frac{k_p-p}{p+1}\frac{p}{k_p-p+1}<1$. These inequalities and \eqref{eq:ineq} imply that 
\[
    \binom{k_p+C}{p-1}\binom{k_p'-C}{p'+1}<\binom{k_p}{p}^2 \Rightarrow \min \left(  \binom{k_p+C}{p-1},\binom{k_p'-C}{p'+1}\right) < r_p.
\]

Then, we consider $p\le \frac{m-2}{2}$. Under the assumption of Case 1, we have 
\begin{eqnarray*}
    \binom{k_p'-C}{p'+1}&=& \binom{k_p'}{p'} \frac{k_p'-p'}{p'+1}\Pi_{i=1}^C \frac{k_p'-i-p'}{k_p'-i+1} \\
    &=&\binom{k_p'-1}{p'} \frac{k_p'}{k_p'-p'} \frac{k_p'-p'}{p'+1}\Pi_{i=1}^C \frac{k_p'-i-p'}{k_p'-i+1} \\
    &<& \binom{k_p}{p} \frac{k_p'}{p'+1} \Pi_{i=1}^C \frac{k_p'-i-p'}{k_p'-i+1}.
\end{eqnarray*}
It holds that $\frac{k_p'}{p'+1}\frac{p}{k_p-p+1}\le 1$ if and only if $k_pm+m+(p-m-n)p\ge 0$. We observe that $k_pm+m+(p-m-n)p$ is increasing in $k_p$ and decreasing in $p$. When $p\le \frac{m-2}{2}$, we have 
\[
    k_pm+m+(p-m-n)p \ge \frac{n-2}{2}m+m+(\frac{m-2}{2}-m-n)\frac{m-2}{2}=n+1-\frac{m^2}{4}.
\]
As a result, $\frac{k_p'}{p'+1}\frac{p}{k_p-p+1}\le 1$ when $n\ge \frac{m^2}{4}-1$. This inequality and \eqref{eq:ineq} imply
\[
    \binom{k_p+C}{p-1}\binom{k_p'-C}{p'+1}<\binom{k_p}{p}^2 \Rightarrow \min \left(  \binom{k_p+C}{p-1},\binom{k_p'-C}{p'+1}\right) < r_p.
\]

For all choices of $p$, the above inequality holds. Under the assumption that $k_{p-1}\ge k_p$, we have $r_{p-1}= \min \left(  \binom{k_p+C}{p-1},\binom{k_p'-C}{p'+1}\right)$ for some $C\ge 0$.
It proves that $r_p>r_{p-1}$ when $n\ge \max \{2m-1,\frac{m^2}{4}-1\}$ under the assumption of Case 1.

\textbf{Case 2:} $\binom{k_p}{p}\le\binom{n-k_p-2}{m-p-1}$. In this case, $r_p=\binom{n-k_p-2}{m-p-1}$. Similar to Case 1, we can show
\[
    \binom{k_p+1+C}{p-1}<\binom{k_p'-1}{p'}\frac{k_p+1}{k_p+1-p}\frac{p}{k_p-p+2}\Pi_{i=1}^C \frac{k_p+1+i}{k_p-p+2+i}
\]
\[
    \binom{k_p'-1-C}{p'+1}<\binom{k_p'-1}{p'}\frac{k_p'-p'-1}{p'+1}\Pi_{i=1}^C \frac{k_p'-p'-1-i}{k_p'-i}
\]
and $ \frac{k_p+1+i}{k_p-p+2+i}\frac{k_p'-p'-1-i}{k_p'-i}\le 1$ for $i\ge -1$. We observe that $\frac{k_p+1}{k_p+1-p}\frac{p}{k_p-p+2}\frac{k_p'-p'-1}{p'+1}$ is decreasing in $k_p$ and increasing in $p$. Recall that $k\ge \frac{n-2}{2},p\le \frac{m-1}{2}$, so we can show
\begin{equation}\label{eq:kp case2}
     \frac{k_p+1}{k_p+1-p}\frac{p}{k_p-p+2}\frac{k_p'-p'-1}{p'+1} \le 1 \Leftarrow 4n^2-\beta n+\gamma\ge 0,
\end{equation}
where $\beta=m^2-2m-8$ and $\gamma=m^3-4m^2-4m+16$.
The inequality on the right of \eqref{eq:kp case2} is quadratic in $n$, so it holds when $\beta^2-16\gamma \le 0$ or
$
    n \ge \frac{\beta+\sqrt{\beta^2-16\gamma}}{8}.
$
The assumption $n\ge \max \{2m-1,\frac{m^2}{4}-1\}$ implies that
\[
    n \ge \frac{m^2}{4}-1 \ge \frac{\beta}{4}\ge \frac{\beta+\sqrt{\beta^2-16\gamma}}{8}.
\]
Therefore, the inequality \eqref{eq:kp case2} holds. Similar to Case 1, it concludes the proof of $r_p>r_{p-1}$ for Case 2.

Summarizing everything above, we prove that \eqref{eq:max r} holds for $n \ge \max\{2m-1,\frac{m^2}{4}-1\}$.
\end{proof}


	
	

The following Theorem \ref{largest-r} provides the largest rank that Algorithm \ref{algo:iSTD} can compute based on the result of Lemma \ref{lemma: lid A}.

\begin{theorem}\label{largest-r}
Let $\cF_m\in S^{m}(\mathbb{C}^{n+1})$ be the tensor as in \eqref{F-decomp}. When $n\geq\max(2m-1,\frac{m^2}{4}-1)$, the largest rank $r$ of $\cF_m$ that Algorithm~\ref{algo:iSTD} can calculate is
\begin{equation}\label{rmax}
    r_{\text{max}}=\max(\binom{ k^*}{p^*}, \binom{n-2-k^*}{m-1-p^*}),
\end{equation}
where $p^*=\lfloor \frac{m-1}{2} \rfloor$ and $k^*$ is largest integer $k$ such that
    $\binom{k}{p^*}\le\binom{n-k-1}{m-p^*-1}$.
\end{theorem}

\begin{proof}
By Theorem \ref{thm:TD}, Algorithm \ref{algo:iSTD} requires $\binom{k}{p}\geq r$ and $\binom{n-k-1}{m-p-1}\geq r$ to find a rank-$r$ decomposition of tensor $\cF_m$. For the given tensor $\cF_m$ with dimension $n+1$ and order $m$, the largest computable rank of Algorithm~\ref{algo:iSTD} is
\begin{equation}\label{max-r}
r_{\max}=\max\limits_{k,p}\left(\min\left(\binom{k}{p}, \binom{n-k-1}{m-p-1}\right)\right),
\end{equation}
where $p\in[1,\, m-2]$ and $k\in[p+1,\, n-m+p-1]$. Therefore, \eqref{rmax} is a direct result of Lemma \ref{lemma:max r proof}.
\end{proof}

\begin{remark}
    The $k^*$ in Theorem \ref{largest-r} can be obtained by solving
    \begin{equation}\label{eq:solve k}
		\binom{k}{p^*}=\binom{n-k-1}{m-1-p^*},
    \end{equation}
    where the above binomial coefficients are generalized to binomial series for real number $k$.
    Let $\tilde{k}\in \mathbb{R}$ be the solution to \eqref{eq:solve k}, then $k^*=\lfloor \tilde{k} \rfloor$. Especially, when $m$ is odd, we have $p^*=m-1-p^*=\frac{m-1}{2}$, $k^*=\lfloor \frac{n-1}{2} \rfloor$, and the corresponding largest rank is
    $$r_{max}=\binom{\lfloor \frac{n-1}{2} \rfloor}{\frac{m-1}{2}}.$$
    There is no uniform formula for the largest ranks when $m$ is even. The largest ranks for some small orders are summarized in Table \ref{tab:max ranks}.
\end{remark}

\begin{table}[h]
    \centering
    \begin{tabular}{c|c}
    \toprule
    $m$ & the largest $r$ \\
    \midrule
    $3$ & $\lfloor\frac{n-1}{2}\rfloor$ \\
    \midrule
    $4$ & $\lfloor\frac{2n-1-\sqrt{8n-7}}{2}\rfloor$ \\
    \midrule
    $5$ & $\binom{\lfloor\frac{n-1}{2}\rfloor}{2}$ \\
    \midrule
    \multirow{2}{*}{$6$} & $\max(\binom{\lfloor \tilde{k}\rfloor}{2}, \binom{n-\lfloor \tilde{k}\rfloor-2}{3})$, where $\Delta=\frac{9}{4}n^4-\frac{47}{2}n^3+\frac{353}{4}n^2-\frac{412}{3}n+\frac{1889}{27}$\\
     & and $\tilde{k}=\sqrt[3]{-\frac{3}{2}(n-3)(n-4)+\sqrt{\Delta}}+\sqrt[3]{-\frac{3}{2}(n-3)(n-4)-\sqrt{\Delta}}+n-3$\\
    \midrule
    $7$ & $\binom{\lfloor\frac{n-1}{2}\rfloor}{3}$ \\
    \bottomrule
    \end{tabular}
    \caption{The largest rank $r$ that Algorithm \ref{algo:iSTD} can compute.}
    \label{tab:max ranks}
\end{table}

\section{Incomplete Tensor Approximations and Error Analysis}
\label{sc:approximation}
When learning Gaussian mixture models, the subtensor $(\cF_m)_{\Omega_m}$ is estimated from samples and is not exactly given. In such case, Algorithm \ref{algo:iSTD} can still find a good low-rank approximation of $\cF_m$. In this section, we discuss how to obtain a good tensor approximation of $\cF_m$ and provide an error analysis for the approximation.

Let $\widehat{\cF}_m$ be approximations of $\cF_m$. Given the subtensor $(\widehat{\cF}_m)_{\Omega_m}$, we can find a low-rank approximation of $\cF_m$ following Algorithm \ref{algo:iSTD}. We define the matrix $A[\alpha, \widehat{\cF}_m]$ and the vector $b[\alpha, \widehat{\cF}_m]$ in the same way as in \reff{eq: Aalpha}, for each $\alpha\in\mathscr{B}_1$. Then we have the following linear least square problem
\begin{equation}\label{solve-G-ls}
    \min_{g_{\alpha} \in \bC^{\sB_0} } \quad
     \left\| A[\alpha,\widehat{\cF}_m] \cdot g_{\alpha}-b[\alpha,\widehat{\cF}_m]\right\|^2.
\end{equation}
For each $\alpha\in\sB_1$, we solve \eqref{solve-G-ls} to get $\widehat{G}[:,\alpha]$ which is an approximation of $G[:,\alpha]$. Combining all $\widehat{G}[:,\alpha]$'s, we get $\widehat{G}\in\mathbb{C}^{\sB_0\times\sB_1}$ approximating the generating matrix $G$. Similar to \eqref{eq:N}, for $l=k+1,\ldots,n$, we define $N_l(\widehat{G})$ as an approximation of $N_l(G)$ and let
\begin{equation}\label{xi-est}
\widehat{N}(\xi)\coloneqq\xi_{k+1}N_{k+1}(\widehat{G})+\cdots+\xi_{n}N_{n}(\widehat{G}),
\end{equation}
where $\xi=(\xi_{k+1},\ldots,\xi_{n})$ is a generic vector. Let $\hat{v}_1,\ldots,\hat{v}_r$ be the unit length eigenvectors of $\widehat{N}(\xi)$ and
\begin{equation}\label{w_hat}
    \hat{w}_i\coloneqq(\hat{v}_i^H N_{k+1}(\widehat{G})\hat{v}_i,\cdots,
    \hat{v}_i^H N_n(\widehat{G})\hat{v}_i)\quad i=1,\ldots,r.
\end{equation}
For the sets $J_1,J_1^{-j},J_2,J_3$ defined in \reff{label-set-J}, we solve the linear least square problem
\begin{equation}\label{coefest-1}
    \min\limits_{(\gamma_1,\ldots,\gamma_r)} \left\|(\widehat{\cF}_{m})_{J_1\cdot J_2}-
     \sum\limits_{i=1}^r\gamma_i \otimes [\hat{w}_i]_{J_{3}} \right\|^2.
\end{equation}
Let $(\hat{\gamma}_1,\ldots,\hat{\gamma}_r)$ be the minimizer of the above problem. Then, we consider the following linear least square problem
\begin{equation}\label{coefest-2}
    \min\limits_{(v_1,\ldots,v_r)} \left\|(\widehat{\cF}_{m})_{x_j \cdot J_1^{-j}\cdot J_2}-\sum\limits_{i=1}^r(v_i)_j\cdot\hat{\gamma}_i\otimes [\hat{w}_i]_{J_{3}} \right\|^2.
\end{equation}
We obtain $(\hat{v}_1,\ldots,\hat{v}_r)$ by solving the above problem for $j=1,\ldots,k$. Let $\hat{u}=(\hat{v},\hat{w})$. Then, we have the following linear least square problem
\begin{equation}\label{coefest-lambda}
    \min\limits_{(\lambda_1,\ldots,\lambda_r)} \left\|(\widehat{\cF}_m)_{\Omega_m}-\sum\limits_{i=1}^r\lambda_i\cdot \left(\begin{bmatrix}
        1\\\hat{u}_r
    \end{bmatrix}^{\otimes m} \right)_{\Omega_m}\right\|^2.
\end{equation}
Denote the minimizer of \reff{coefest-lambda} as $(\hat{\lambda}_1,\ldots,\hat{\lambda}_r)$.
For $i=1,\ldots,r$, let
\begin{equation*}
\hat{q}_i\coloneqq(\hat{\lambda}_i)^{1/m}(1,\hat{v}_i,\hat{w}_i).
\end{equation*}
Now, we obtain the approximation of the tensor $\cF_m$
\begin{equation*}
    \cF_m\approx(\hat{q}_1)^{\otimes m}+\cdots+(\hat{q}_r)^{\otimes m}.
\end{equation*}
This result may not be optimal due to sample errors. We can get a more accurate approximation by using $(\hat{q}_1,\ldots,\hat{q}_r)$ as starting points to solve the nonlinear optimization
\begin{equation} \label{solve-F}
\min\limits_{(q_1,\ldots,q_r)} \, \left\|
(\widehat{\cF}_m)_{\Omega_{m}}- \sum_{i=1}^r (q_i^{\otimes m})_{\Omega_{m}}\right\|^2.
\end{equation}
We denote the minimizer of the optimization \reff{solve-F} as $(q_1^*,\ldots,q_r^*)$.

We summarize the above calculations as a tensor approximation algorithm in Algorithm \ref{algo:tensor-approx}.

\begin{alg}  \label{algo:tensor-approx}
(Incomplete symmetric tensor approximation.)
\begin{itemize}

\item [Input:] The rank $r$, the dimension $d$, the constant $p$, and the subtensor $(\widehat{\cF}_m)_{\Omega_m}$ as in \eqref{subtensor-approx}.

\item [Step~1.] Determine the generating matrix $\widehat{G}$ by solving \reff{solve-G-ls} for each $\alpha\in\mathscr{B}_1$.

\item [Step~2.] Choose a generic vector $\xi$ and define $\widehat{N}(\xi)$ as in \eqref{xi-est}. Calculate unit length eigenvectors of $\widehat{N}(\xi)$ and corresponding eigenvalues of each $N_i(\widehat{G})$ to define $\hat{w}_i$ as in \eqref{w_hat}.

\item [Step~3.] Solve \eqref{coefest-1}, \eqref{coefest-2} and \eqref{coefest-lambda} to obtain the coefficients $\hat{\lambda}_i$ and vectors $\hat{v}_i$.

\item[Step~4.] Let $\hat{q}_i\coloneqq(\hat{\lambda}_i)^{1/m}(1,\hat{v}_i,\hat{w}_i)$ for $i=1,\ldots,r$. Use $\hat{q}_1,\ldots,\hat{q}_r$ as start points to solve the nonlinear optimization \eqref{solve-F} and get an optimizer $(q_1^*,\ldots,q_r^*)$.

\item [Output:] The incomplete tensor approximation $(q_1^*)^{\otimes m}+\cdots+(q_r^*)^{\otimes m}$ for $\widehat{\cF}_m$.

\end{itemize}
\end{alg}

We can show that Algorithm~\ref{algo:tensor-approx} provides a good rank-$r$ approximation when the input subtensor $(\widehat{\cF}_m)_{\Omega_m}$ is close to exact tensors $\cF_m$.

\begin{theorem}\label{thm:tensor}
Let $\cF_m = \omega_1(\mu_1)^{\otimes m}+\cdots+\omega_r(\mu_r)^{\otimes m}$ as in \eqref{eq:cFm} and constants $k,\,p$ be such that $\min\left(\binom{k}{p}, \binom{n-k-1}{m-p-1}\right)\geq r$. We assume the following conditions:
\begin{itemize}
\item [(i)] the scalars $\omega_i$ and the leading entry of each $\mu_i$ are nonzero;

\item [(ii)] the vectors $\{[(\mu_i)_{1:n}]_{\sB_0}\}_{i=1}^r$ are linearly independent;

\item [(iii)] the vectors $\{[(\mu_i)_{1:n}]_{\mathcal{O}_\alpha}\}_{i=1}^r$ are linearly independent for all $\alpha \in \sB_1$;

\item [(iv)] the eigenvalues of the matrix $N(\xi)$ in \reff{N_xi}
are distinct from each other.
\end{itemize}
Let $q_i=(\omega_i)^{1/m}\mu_i$ and $q_i^*$ be the output vectors of
Algorithm~\ref{algo:tensor-approx}.
If the distance $\epsilon := \|(\widehat{\cF}_m-\cF_m)_{\cI_m} \|$ is small enough, then there exist scalars $\tilde{\eta}_i, \eta_i^{*}$ such that
\[
(\hat{\eta}_i)^{m} = (\eta_i^{*})^{m}=1, \quad
\|\hat{\eta}_i\hat{q}_i- q_i\| = O(\epsilon), \quad
\|\eta_i^{*}{q}^{*}_i- q_i\| = O(\epsilon),
\]
up to a permutation of $(q_1, \ldots, q_r)$,
where the constants inside $O(\cdot)$ only depend on $\cF_m$ and the choice of $\xi$ in Algorithm~\ref{algo:tensor-approx}.
\end{theorem}

\begin{proof}
The vectors $(1,u_1),\,\ldots,\,(1,u_r)$ in \eqref{F-decomp} are scalar multiples of $\mu_1,\,\ldots,\mu_r$ respectively. By Conditions (ii) and (iii), the vectors $\{[u_i]_{\sB_0}\}_{i=1}^r$ and $\{[u_i]_{\cO_\alpha}\}_{i=1}^r$ are both linearly independent, which satisfies the condition of Lemma~\ref{lemma: lid A}. Thus Conditions (i)-(iii) imply that there exists a unique generating matrix $G$ for $\cF_m$ by Theorem~\ref{thm: G} and it can be calculated by \reff{solve-G}. By Lemma~\ref{lemma: lid A}, the matrix $A[\alpha, \cF_m]$ has full column rank. It holds that
\begin{equation}
    \|A[\alpha, \cF_m]-A[\alpha, \widehat{\cF}_m]\|\leq\epsilon,\,\|b[\alpha, \cF_m]-b[\alpha, \widehat{\cF}_m]\|\leq\epsilon,
\end{equation}
for $\alpha\in\sB_1$. When $\epsilon$ is small enough, the matrix $A[\alpha, \widehat{\cF}_m]$ also has full column rank. Then the linear least square problems \eqref{solve-G-ls} have unique solutions and the collection of solutions $\widehat{G}$ satisfies that
\begin{equation*}
    \|G-\widehat{G}\|=O(\epsilon),
\end{equation*}
where $O(\epsilon)$ depends on $\cF_m$ (see \cite[Theorem~3.4]{Demmel}). Since $N_l(\widehat{G})$ is part of the generating matrix $\widehat{G}$ for each $l=k+1,\ldots,n$, we have
\begin{equation*}
    \|N_l(\widehat{G})-{N}_l(G)\|\le \|\widehat{G}-{G}\|=O(\epsilon), \quad l=k+1,\ldots,n,
\end{equation*}
which implies that $\|\widehat{N}(\xi)-{N}(\xi)\|=O(\epsilon) $. By condition (iv) we know that the matrix $\widehat{N}(\xi)$ has distinct eigenvalues $\hat{w}_1,\ldots,\hat{w}_r$ if $\epsilon$ is small enough. So the matrix $N(\xi)$ has a set of  eigenvalues $\tilde{w}_i$ such that
\begin{equation*}
    \|\hat{w}_i-\tilde{w}_i\|=O(\epsilon).
\end{equation*}
This follows from Proposition 4.2.1 in \cite{chatelin2012}. The constants inside the above $O(\cdot)$ depend only on $\cF_m$ and $\xi$. The vectors $\tilde{w}_1,\ldots,\tilde{w}_r$ are multiples of the vectors
$
(\mu_1)_{k+1:n},\ldots,(\mu_r)_{k+1:n}
 $
 respectively. Thus, we conclude that $[\tilde{w}_1]_{J_{3}},\ldots,[\tilde{w}_r]_{J_{3}}$ are linearly independent by condition (iii). When $\epsilon$ is small, the vectors
 $
 [\hat{w}_1]_{J_{3}},\ldots,[\hat{w}_r]_{J_{3}}
  $
  are also linearly independent. For optimizers $\hat{\gamma}_i,\,\hat{v}_i,\,\hat{\lambda}_i$ of linear least square problems \eqref{coefest-1}, \eqref{coefest-2} and \eqref{coefest-lambda}, by \cite[Theorem~3.4]{Demmel}, we have
\begin{equation*}
    \|\hat{\gamma}_i-\gamma_i\|=O(\epsilon),\,\|\hat{v}_i-v_i\|=O(\epsilon),\,\|\hat{\lambda}_i-\lambda_i\|=O(\epsilon),
\end{equation*}
where constants inside $O(\cdot)$ depend on $\cF_m$ and $\xi$. By Theorem~\ref{thm:TD}, we have $\cF_m=\sum_{i=1}^r\tilde{q}_i^{\otimes m}$ where $\tilde{q}_i=(\tilde{\lambda})^{1/m}(1,\tilde{v}_i,\tilde{w}_i)$.The rank-$r$ decomposition of $\cF_m$ is unique up to scaling and permutation by Theorem \ref{thm:TD}. Thus, there exist scalars $\hat{\eta}_i$ such that $(\hat{\eta}_i)^m=1$ and $\hat{\eta}_i\tilde{q}_i=q_i$, up to a permutation of $q_1,\ldots,q_r$. Then for $\hat{q}_i=(\hat{\lambda})^{1/m}(1,\hat{v}_i,\hat{w}_i)$, we have $\|\hat{\eta}_i\hat{q}_i-q_i\|=O(\epsilon)$ where constants inside $O(\cdot)$ depend on $\cF_m$ and $\xi$.

Since $\|\hat{\eta}_i\hat{q}_i-q_i\|=O(\epsilon)$, we have $\|\cF_m-(\sum_{i=1}^r (\hat{q}_i)^{\otimes m})_{\Omega_m}\|=O(\epsilon)$. For the minimizer $(q_1^*,\ldots,q_r^*)$ of \eqref{solve-F}, it holds that
\begin{equation*}
\left\|\left(\widehat{\cF}_m-\sum_{i=1}^r(q_i^{*})^{\otimes m}
\right)_{\Omega_m}\right\|\le
\left\|\left(\widehat{\cF}_m-\sum_{i=1}^r (\hat{q}_i)^{\otimes m}\right)_{\Omega_m}\right\| = O(\epsilon).
\end{equation*}
For the tensor $\cF_m^{*}:=\sum_{i=1}^r (q_i^{*})^{\otimes m}$, we have
\begin{equation*}
    \|(\cF_m^{*}-\cF_m)_{\Omega_m}\|\le\|(\cF_m^{*}-\widehat{\cF}_m)_{\Omega_m}\|+\|(\widehat{\cF}_m-\cF_m)_{\Omega_m}\|=O(\epsilon)
\end{equation*}

If we apply Algorithm~\ref{algo:tensor-approx} to $(\cF_m^*)_{\Omega_m}$, we will get the exact decomposition $\cF_m^{*}=\sum_{i=1}^r (q_i^{*})^{\otimes m}$. By repeating the above argument, similarly we can obtain that $\|\eta_i^{*}q_i^{*}-q_i\|=O(\epsilon)$ for some $\eta_i^*$ such that $(\eta_i^*)^m=1$, where the constants in $O(\cdot)$ only depend on $\cF_m$ and $\xi$.
\end{proof}

\section{Learning General Diagonal Gaussian Mixture}\label{sc:gmm}

Let $y$ be the random variable of a diagonal Gaussian mixture model and 
$y_1,\ldots,y_N$ be i.i.d. samples drawn from the model. 
The moment tensors $M_m\coloneqq\mathbb{E}[y^{\otimes m}]$ can be estimated as follows
\begin{align*}
\widehat{M}_m&\coloneqq\frac{1}{N}(y_1^{\otimes m} + \cdots + y_N^{\otimes m}).
\end{align*}
Recall that $\cF_m=\sum_{i=1}^r\omega_i\mu_i^{\otimes m}$. By Corollary~\ref{moment-structure} and \eqref{high-tensor-moment}, we have
\begin{align*}
    (M_m)_{\Omega_m}&=(\cF_m)_{\Omega_m},
\end{align*}
 where $\Omega_{m}$ is the index set defined in \reff{index-set}. Let $\widehat{\cF}_{m}$ be such that
\begin{equation}\label{subtensor-approx}(\widehat{\cF}_m)_{\Omega_m}\coloneqq(\widehat{M}_m)_{\Omega_m}.
\end{equation}
We can apply Algorithm \ref{algo:tensor-approx} to find the low-rank approximation of $\widehat{\cF}_m$. Let $\widehat{\cF}_m\approx \sum_{i=1}^r (q_i^*)^{\otimes m}$ be the tensor approximation generated by Algorithm \ref{algo:tensor-approx}. By Theorem \ref{thm:tensor}, when $\epsilon=\|(\widehat{M}_m)_{\Omega_m}-({M}_m)_{\Omega_m}\|$ is small, there exists $\eta_i \in \bC$ such that $\eta_i^{m}=1$ and $\|\eta_i q_i^*-(\omega_i)^{1/m}\mu_i\|=O(\epsilon)$. The $\eta_i$ appears here because the vector $q_i^*$ can be complex even though $\widehat{\cF}_m$ is a real tensor. But $\omega_i, \mu_i$ are both real in Gaussian mixture models. In practice, we can choose the $\eta_i$ from all $m$th roots of $1$ that minimizes $\|\imag(\eta_iq_i^*)\|$. Let
\begin{equation} \label{eq:q}
    \check{q}_i \coloneqq \real(\eta_i q_i^*).
\end{equation}
We expect that $\check{q}_i \approx (\omega_i)^{1/m}\mu_i $. Then, we consider the tensor
\begin{equation*}
{\cF}_{t}=\omega_1\mu_1^{\otimes t}+\cdots+\omega_r\mu_r^{\otimes t}\approx(\omega_1)^{\frac{m-t}{m}}(\check{q}_1)^{\otimes t}+\cdots+(\omega_r)^{\frac{m-t}{m}}(\check{q}_r)^{\otimes t},
\end{equation*}
where $t$ is the smallest number such that $\binom{d}{t}\ge r$. It holds that $({\cF}_{t})_{\Omega_{t}}=({M}_{t})_{\Omega_{t}}\approx(\widehat{M}_{t})_{\Omega_{t}}$, so we obtain the scalars $(\omega_i)^{\frac{m-t}{m}}$ by solving the linear least square problem
\begin{equation} \label{solve-w}
\min\limits_{(\beta_1,\ldots,\beta_r)\in\mathbb{R}^r_+} \, \left\|
(\widehat{M}_{t})_{\Omega_{t}}-\sum_{i=1}^r\beta_i\left((\check{q}_i)^{\otimes t}\right)_{\Omega_{t}}\right\|^2.
\end{equation}
Let the optimizer of \reff{solve-w} be $(\beta_1^{\ast},\ldots,\beta_r^{\ast})$, then
\begin{equation}\label{est-w-mu}
    \hat{\omega}_i=(\beta^*)^{\frac{m}{m-t}}\quad\text{and}\quad\hat{\mu}_i=q^*_i/(\beta_i^*)^{\frac{1}{m-t}}
\end{equation}
should be reasonable approximations of $\omega_i$ and $\mu_i$ respectively.

To obtain more accurate results, we can use $(\hat{\omega}_1,\ldots,\hat{\omega}_r, \hat{\mu}_1,\ldots, \hat{\mu}_r)$ as starting points to solve the following nonlinear optimization
\begin{equation}\label{prob:nls-w mu}
\left\{\begin{array}{cl}
\min\limits_{ \substack{\omega_1,\ldots,\omega_r, \\ \mu_1,\ldots,\mu_r}  }&
\|(\widehat{M}_m)_{\Omega_m}-\sum\limits_{i=1}^r \omega_i (\mu_i^{\otimes m})_{\Omega_m}\|^2+\|(\widehat{M}_t)_{\Omega_t}-\sum\limits_{i=1}^r \omega_i (\mu_i^{\otimes t})_{\Omega_t}\|^2\\
\text{subject to}&\omega_1+\cdots+\omega_r=1,\,
     \omega_1,\ldots,\omega_r \ge 0,
\end{array} \right.
\end{equation}
and obtain the optimizer $(\omega_1^*,\ldots,\omega_r^*, \mu_1^*,\ldots, \mu_r^*)$.

Next, we will show how to calculate the diagonal covariance matrices. We define a label set
\begin{equation*}
	L_j=\{(j,j,i_1,\ldots,i_{m-2}):1\leq i_1<\cdots<i_{m-2}\leq d, \text{and } i_1\ne j,\ldots,i_{m-2}\ne j\}.
\end{equation*}
For $(j,j,i_1,\ldots,i_{m-2})\in L_j$, we have
\begin{equation*}
(M_{m})_{j,j,i_1,\ldots,i_{m-2}}=\sum\limits_{i=1}^r\omega_i\left((\mu_i)_j(\mu_i)_j(\mu_i)_{i_1}\cdots(\mu_i)_{i_{m-2}}+\Sigma^{(i)}_{jj}(\mu_i)_{i_1}\cdots(\mu_i)_{i_{m-2}}\right).
\end{equation*}
The above equation is a direct result of \reff{moment-structure} since all covariance matrices are diagonal. Let
\begin{equation}\label{cov-A}
\mathcal{A}\coloneqq M_m-\cF_m,\quad
\widehat{\mathcal{A}}\coloneqq\widehat{M}_m-(\check{q}_1)^{\otimes m}-\cdots-(\check{q}_r)^{\otimes m}.
\end{equation}
To get the estimation of covariance matrices $\Sigma_i=\text{diag}(\sigma_{i1}^2,\ldots,\sigma_{id}^2)$, we solve the nonnegative linear least square problems $(j=1,\ldots,d)$

\begin{equation}\label{solve-sigma}
\left\{ \baray{cl}
 \min\limits_{(\theta_{1j},\ldots,\theta_{rj})} & \left\|\left(\widehat{\mathcal{A}}\right)_{L_j}-    \sum\limits_{i=1}^r\theta_{ij}\omega^*_i \left((\mu_i^*)^{\otimes m-2}\right)_{\hat{L}_j}\right\|^2\\
 \mbox{subject to} & \theta_{1j} \ge 0,\ldots,\,\theta_{rj} \ge 0
\earay \right.
\end{equation}
where $\hat{L}_j=\{(i_1,\ldots,i_{m-2}):(j,j,i_1,\ldots,i_{m-2}) \in L_j\}$. The vector $((\mu_i^*)^{\otimes m-2})_{\hat{L}_j}$ has length $\binom{n}{m-2} \ge \binom{k}{p}\ge r$, where $k,p$ are constants in Algorithm \ref{algo:tensor-approx}. Therefore, $((\mu_1^*)^{\otimes m-2})_{\hat{L}_j},\ldots,((\mu_r^*)^{\otimes m-2})_{\hat{L}_j}$ are generically linearly independent and hence \eqref{solve-sigma} has a unique optimizer. Suppose the optimizer is $(\theta_{1j}^*,\ldots,\theta_{rj}^*)$. The covariance matrix $\hat{\Sigma}_i$ can be approximated as
\begin{equation}\label{Sig_i^opt}
\Sigma^*_i\coloneqq\{\diag((\theta_{i1}^*,\ldots,\theta_{id}^*))\},\, \sigma_{ij}^*\coloneqq \sqrt{\theta_{ij}^*}.	
\end{equation}

The following is the complete algorithm to recover the unknown parameters $\{\mu_i,\Sigma_i,\Omega_i\}_{i=1}^r$.

\begin{alg}  \label{algo:gaussian}
(Learning diagonal Gaussian mixture models.)
\begin{itemize}

\item [Input:] The $m$th order sample moment tensor $\widehat{M}_{m}$, the $t$th order sample moment tensor $\widehat{M}_{t}$, and the number of components $r$.

\item [Step~1.] Apply Algorithm~\ref{algo:tensor-approx} to subtensor $(\widehat{\cF}_m)_{\Omega_m}$ defined in \eqref{subtensor-approx}. Let $(q_1^*)^{\otimes m}+\cdots+(q_r^*)^{\otimes m}$ be the output incomplete tensor approximation for $\widehat{\cF}_m$.

\item [Step~2.] For $i=1,\ldots,r$, we choose $\eta_i$ such that $\eta_i^m=1$ and it minimizes $\|\text{Im}(\eta_i q_i^*)\|$. Let $\check{q}_i=\text{Re}(\eta_i q_i^*)$ as in \eqref{eq:q}.

\item [Step~3.] Solve \reff{solve-w} to get the optimizer $(\beta_1^*,\ldots,\beta_r^*)$ and compute $\hat{\omega}_i$, $\hat{\mu}_i$ as in \eqref{est-w-mu} for $i=1,\ldots,r$.

\item [Step~4.] Use $(\hat{\omega}_1,\ldots,\hat{\omega}_r, \hat{\mu}_1,\ldots, \hat{\mu}_r)$ as starting points to solve \eqref{prob:nls-w mu} to obtain the optimizer $(\omega_1^*,\ldots,\omega_r^*, \mu_1^*,\ldots, \mu_r^*)$.

\item [Step~5.] Solve the optimization \reff{solve-sigma} to get optimizers $\theta_{ij}^*$
and then compute $\Sig^*_i$ as in \reff{Sig_i^opt}.

\item [Output:] Mixture Gaussian parameters
$(\omega^*_i,\mu^*_i,\Sigma^*_i), i=1,\ldots,r$.

\end{itemize}
\end{alg}


When the sampled moment tensors are close to the accurate moment tensors, the parameters generated by Algorithm \ref{algo:gaussian} are close to the true model parameters. The analysis is shown in the following theorem.

\begin{theorem} \label{thm:gaussian approx}
   Given a $d$-dimensional diagonal Gaussian mixture model with parameters $\{(\omega_i,\mu_i,\Sigma_i):i\in[r]\}$ and $r$ no greater than the $r_{\max}$ in \eqref{rmax}. Let $\{(\omega^{*}_i,\mu^{*}_i,\Sigma^{*}_i):i\in[r]\}$ be the output of Algorithm~\ref{algo:gaussian}. If the distance $\epsilon :=\max(\|\widehat{M}_m-M_m\|,\|\widehat{M}_{t}-M_{t}\|) $ is small enough, $(\mu_1^{\otimes t})_{\Omega_t},\ldots,(\mu_r^{\otimes t})_{\Omega_t}$ are linearly independent, and the tensor $\cF_m=\sum_{i=1}^r\omega_i\mu_i^{\otimes m}$ satisfies the conditions of Theorem~\ref{thm:tensor}, then
   \begin{equation*}
       \|\mu^{*}_i-\mu_i\| = O(\epsilon), \|\omega_i^{*}-\omega_i\| = O(\epsilon),
\|\Sigma^{*}_{i}-\Sigma_i\| = O(\epsilon),
   \end{equation*}
   where the constants inside $O(\cdot)$ depend on parameters $\{(\omega_i,\mu_i,\Sigma_i):i\in[r]\}$ and the choice of $\xi$ in Algorithm~\ref{algo:gaussian}.
\end{theorem}

\begin{proof}
    We have
    \begin{gather*}
        \|(\widehat{\cF}_m-\cF_m)_{\Omega_m}\|=\|(\widehat{M}_m-M_m)_{\Omega_m}\| \le \epsilon,\\
        \|(\widehat{\cF}_{t}-\cF_{t})_{\Omega_{t}}\|=\|(\widehat{M}_{t}-M_{t})_{\Omega_{t}}\|\le\epsilon.
    \end{gather*}
    and $\cF_m,\,\cF_{t}$ satisfy conditions of Theorem~\ref{thm:tensor}. They imply that $\|\eta_i^*q_i^*-q_i\|=O(\epsilon)$ for some $(\eta_i^*)^m=1$ by Theorem~\ref{thm:tensor}. The constants inside $O(\epsilon)$ depend on the parameters of the Gaussian model and vector $\xi$. Since vectors $q_i$ are real, we have $\|\text{Im}(\eta_i^*q_i^*)\|=O(\epsilon)$. When $\epsilon$ is small enough, such $\eta_i^*$ minimizes $\|\text{Im}(\eta_i^*q_i^*)\|$ and we have
    \begin{equation*}
        \|\text{Re}(\eta_i^*q_i^*)-q_i\|\leq\|\eta_i^*q_i^*-q_i\|=O(\epsilon).
    \end{equation*}
    Let $\check{q}_i\coloneqq\text{Re}(\eta_i^*q_i^*)$. When $\epsilon$ is small, vectors $(\check{q}_1^{\otimes t})_{\Omega_t},\ldots,(\check{q}_r^{\otimes t})_{\Omega_t}$ are linearly independent since $(\mu_1^{\otimes t})_{\Omega_t},\ldots,(\mu_r^{\otimes t})_{\Omega_t}$ are linearly independent by our assumption. It implies that the problem \eqref{solve-w} has a unique solution. The weights $\hat{\omega}_i$ and mean vectors $\hat{\mu}_i$ can be calculated by \eqref{est-w-mu}. Since $\|(\widehat{M}_{t}-M_{t})_{\Omega_{t}}\|\le\epsilon$ and $\|\check{q}_i-q_i\|=O(\epsilon)$, we have $\|\omega_i - \hat{ \omega }_i \|=O(\epsilon)$ (see \cite[Theorem~3.4]{Demmel}). The approximation error for the mean vectors is
    \begin{equation*}
        \|\hat{\mu}_i-\mu_i\|=\|\check{q}_i/(\hat{\omega}_i)^{1/m}-{q}_i/(\omega_i)^{1/m}\| = O(\epsilon).
    \end{equation*}
    The constants inside $O(\epsilon)$ depend on parameters of the Gaussian mixture model and $\xi$.

    We obtain optimizers $\omega_i$ and $\mu_i$ by solving the problem \eqref{prob:nls-w mu}, so it holds
    \begin{equation*}
        \left\|(\widehat{M}_m)_{\Omega_m}-\sum\limits_{i=1}^r \omega^*_i \left((\mu^*_i)^{\otimes m}\right)_{\Omega_m}\right\|=O(\epsilon).
    \end{equation*}
    Let $\cF_m^*:=\sum_{i=1}^r \omega_i^{*} (\mu_i^{*})^{\otimes m}$ and $\cF_{t}^*:=\sum_{i=1}^r \omega_i^{*} (\mu_i^{*})^{\otimes t}$, then
    \begin{gather*}
    \|(\cF_m^*-\cF_m)_{\Omega_m}\|\le\|(\widehat{\cF}_m-\cF_m)_{\Omega_m}\|+\|(\widehat{\cF}_m-\cF_m^*)_{\Omega_m}\| = O(\epsilon).\\
    \|(\cF_{t}^*-\cF_{t})_{\Omega_{t}}\|\le\|(\widehat{\cF}_{t}-\cF_{t})_{\Omega_{t}}\|+\|(\widehat{\cF}_{t}-\cF_{t}^*)_{\Omega_{t}}\| = O(\epsilon).
    \end{gather*}
    By Theorem~\ref{thm:tensor}, we have $\|(\omega_i^*)^{1/m}\mu_i^*-q_i\|=O(\epsilon)$. Since we are optimizing \eqref{prob:nls-w mu}, it also holds that
    \begin{equation*}
       \left\| (\widehat{M}_{t})_{\Omega_t}-\sum_{i=1}^r  \omega_i \left((\mu_i^*)^{\otimes t} \right)_{\Omega_t} \right\|= \left\| (\widehat{M}_{t})_{\Omega_t}-\sum_{i=1}^r  (\omega_i^*)^{\frac{m-t}{m}} \left(((\omega_i^*)^{1/m}\mu_i^*)^{\otimes t} \right)_{\Omega_t} \right\|=O(\epsilon).
    \end{equation*}
    Combining the above with $\|(\widehat{M}_{t}-M_{t})_{\Omega_{t}}\|=O(\epsilon)$, we get $\|(\omega_i^*)^{\frac{m-t}{m}}-\omega_i^{\frac{m-t}{m}}\|=O(\epsilon)$ by \cite[Theorem~3.4]{Demmel} and hence $\|\omega_i^*-\omega_i\|=O(\epsilon)$. For mean vectors $\mu_i$ we have
    \begin{equation*}
        \|\mu_i^*-\mu_i\|=\|((\omega_i^*)^{1/m}\mu_i^*)/(\omega_i^*)^{1/m}-q_i/(\omega_i)^{1/m}\|=O(\epsilon).
    \end{equation*}
    The constants inside the above $O(\cdot)$ only depend on parameters $\{(\omega_i,\mu_i,\Sigma_i):i\in[r]\}$ and $\xi$.

    We obtain the covariance matrices $\Sigma_i$ by solving \eqref{solve-sigma}. It holds that
    \begin{gather*}
        \|\omega^*_i(\mu_i^*)^{\otimes(m-2)}-\omega_i\mu_i^{\otimes(m-2)}\|=O(\epsilon),\\
        \|\widehat{\mathcal{A}}-\mathcal{A}\|\leq\|\widehat{M}_m-M_m\|+\|\sum\limits_{i=1}^r(q_i^*)^{m}-\cF_m\|\leq O(\epsilon),
    \end{gather*}
    where $\widehat{\mathcal{A}}$ and $\mathcal{A}$ are defined in \eqref{cov-A}. The tensor $\cF_m$ satisfies the condition of Theorem \ref{thm:tensor}, so the tensors $\mu_1^{\otimes(m-2)},\ldots,\mu_r^{\otimes(m-2)}$ are linearly independent. It implies that $\{\omega^*_i(\mu_i^*)^{\otimes(m-2)}\}_{i=1}^r$ are linearly independent when $\epsilon$ is small. Therefore, \eqref{solve-sigma} has a unique solution for each $j$. By \cite[Theorem~3.4]{Demmel}, we have
    \begin{equation*}
    \|(\sigma_{ij}^*)^2-(\sigma_{ij})^2\| = O(\epsilon).
    \end{equation*}
    It implies that $\|\Sigma_i^*-\Sigma_i \| = O(\epsilon)$, where the constants inside $O(\cdot)$ only depend on parameters $\{(\omega_i,\mu_i,\Sigma_i):i\in[r]\}$ and $\xi$.
\end{proof}

\begin{remark}
Given the dimension $d$ and the highest order of moment $m$, 
the largest number of components in the Gaussian mixture model that 
Algorithm~\ref{algo:gaussian} can learn is the same as 
the largest rank $r_{max}$ as in Theorem \ref{largest-r}, i.e.,
\begin{equation*}
r_{\text{max}}=\max(\binom{ k^*}{p^*}, \binom{d-3-k^*}{m-1-p^*}),
\end{equation*}
where $p^*=\lfloor \frac{m-1}{2} \rfloor$ and $k^*$ is largest integer $k$ such that
\[  \binom{k}{p^*}\le\binom{d-k-2}{m-p^*-1} . \]
Given a $d$-dimensional Gaussian mixture model with $r$ components, 
we can use Theorem \ref{largest-r} to obtain the smallest order $m$ 
required for the Algorithm \ref{algo:gaussian} and then apply 
Algorithm~\ref{algo:gaussian} to learn the Gaussian mixture model using the $m$th order moment.
\end{remark}

\section{Numerical Experiments}\label{sc:exp}
This section provides numerical experiments for our proposed methods. 
The computation is implemented in {\tt MATLAB} R2023a,
on a personal computer with Intel(R)Core(TM)i7-9700K CPU@3.60GHz and RAM 16.0G. The {\tt MATLAB} function \texttt{lsqnonlin} is used to solve \eqref{solve-F}
in Algorithm~\ref{algo:tensor-approx} and the {\tt MATLAB} function \texttt{fmincon}
is used to solve \eqref{prob:nls-w mu} in Algorithm~\ref{algo:gaussian}.

\subsection{Incomplete tensor decomposition}
In this subsection, we present numerical experiments for Algorithm \ref{algo:iSTD}. We construct
\begin{equation}\label{eq:randomF}
    \cF_m=\sum_{i=1}^r q_i^{\otimes m} \in \textrm{S}^m(\bC^{d})
\end{equation}
by randomly generating each $q_i\in\mathbb{R}^d$ in Gaussian distribution by the {\tt randn} function in {\tt MATLAB}. Then we apply Algorithm \ref{algo:iSTD} to the subtensor $(\cF_m)_{\Omega_m}$ to calculate the rank-$r$ tensor decomposition. The relative error of tensors and components are used to measure the decomposition result
\begin{equation*}
\text{decomp-err}_m\coloneqq\frac{\|(\cF_m-\widetilde{\cF}_m)_{\Omega_{m}}\|}{\|(\cF_m)_{\Omega_m}\|},\quad\text{vec-err-max}\coloneqq\max\limits_i\frac{\|q_i-\tilde{q}_i\|}{\|q_i\|},
\end{equation*}
where $\widetilde{\cF}_m$, $\tilde{q}_i$ are output of Algorithm \ref{algo:iSTD}. We choose the values of $d, m$ as
\begin{equation*}
    d = 15,\,25,\,30,\,40,\quad m = 3,\,4,\,5,\,6,\,7,
\end{equation*}
and $r$ as largest computable rank in Theorem~\ref{largest-r} given $d$ and $m$. For each $(d, m, r)$, we generate $100$ random instances, except for the case $(40, 7, 969)$ where $20$ instances are generated due to the long computation time. The min, average, and max relative errors of tensors for each dimension $d$, order $m$, and the average relative errors of component vectors are shown in Table~\ref{tensor-result}. The results show that Algorithm~\ref{algo:iSTD} finds the correct decomposition of randomly generated tensors.

{
    
\begin{table}[htbp]
  \centering
  \caption{The performance of Algorithm~\ref{algo:iSTD}}\label{tensor-result}
    \scalebox{0.8}{
    \begin{tabular}{ccccccc}
    \toprule
       & & & & $\text{decomp-err}_{m}$ &  \\
    \cmidrule{4-6}
     d & m & r & min & average & max & vec-err-max \\
    \midrule
    \multirow{7}{*}{$15$} & $3$ & $6$ & $1.7\cdot 10^{-15}$ & $3.1\cdot 10^{-12}$ & $1.7\cdot 10^{-10}$ & $1.1\cdot 10^{-11}$ \\
    \cmidrule{4-6}
     & $4$ & $8$ & $4.0\cdot 10^{-15}$ & $7.8\cdot 10^{-10}$ & $7.7\cdot 10^{-8}$ & $1.2\cdot 10^{-10}$ \\
     \cmidrule{4-6}
     & $5$ & $15$ & $1.9\cdot 10^{-14}$ & $2.5\cdot 10^{-11}$ & $8.7\cdot 10^{-10}$ & $9.1\cdot 10^{-11}$ \\
     \cmidrule{4-6}
     & $6$ & $20$ & $5.2\cdot 10^{-13}$ & $2.3\cdot 10^{-10}$ & $1.2\cdot 10^{-8}$ & $9.5\cdot 10^{-10}$ \\
     \cmidrule{4-6}
     & $7$ & $20$ & $7.4\cdot 10^{-14}$ & $1.7\cdot 10^{-10}$ & $1.3\cdot 10^{-8}$ & $3.4\cdot 10^{-10}$\\
     \midrule
    \multirow{7}{*}{$25$} & $3$ & $11$ & $9.3\cdot 10^{-15}$ & $7.3\cdot 10^{-12}$ & $6.3\cdot 10^{-10}$ & $1.3\cdot 10^{-11}$ \\
    \cmidrule{4-6}
     & $4$ & $16$ & $6.1\cdot 10^{-14}$ & $1.0\cdot 10^{-10}$ & $9.1\cdot 10^{-9}$ & $3.5\cdot 10^{-10}$ \\
     \cmidrule{4-6}
     & $5$ & $55$ & $2.9\cdot 10^{-12}$ & $4.4\cdot 10^{-9}$ & $1.2\cdot 10^{-7}$ & $3.8\cdot 10^{-8}$ \\
     \cmidrule{4-6}
     & $6$ & $84$ & $9.3\cdot 10^{-11}$ & $7.2\cdot 10^{-8}$ & $1.7\cdot 10^{-6}$ & $4.6\cdot 10^{-7}$ \\
     \cmidrule{4-6}
     & $7$ & $165$ & $1.4\cdot 10^{-10}$ & $1.4\cdot 10^{-7}$ & $4.1\cdot 10^{-6}$ & $1.7\cdot 10^{-6}$ \\
     \midrule
     \multirow{7}{*}{$30$} & $3$ & $14$ & $3.2\cdot 10^{-14}$ & $7.1\cdot 10^{-12}$ & $2.2\cdot 10^{-10}$ & $4.3\cdot 10^{-11}$ \\
     \cmidrule{4-6}
     & $4$ & $21$ & $3.6\cdot 10^{-13}$ & $1.6\cdot 10^{-10}$ & $2.4\cdot 10^{-8}$ & $3.9\cdot 10^{-10}$ \\
     \cmidrule{4-6}
     & $5$ & $91$ & $3.3\cdot 10^{-11}$ & $1.5\cdot 10^{-7}$ & $5.3\cdot 10^{-6}$ & $6.2\cdot 10^{-7}$ \\
     \cmidrule{4-6}
     & $6$ & $136$ & $1.0\cdot 10^{-10}$ & $1.7\cdot 10^{-7}$ & $8.3\cdot 10^{-6}$ & $1.9\cdot 10^{-6}$\\
     \cmidrule{4-6}
     & $7$ & $364$ & $2.4\cdot 10^{-8}$ & $2.4\cdot 10^{-6}$ & $2.7\cdot 10^{-5}$ & $4.1\cdot 10^{-5}$ \\
     \midrule
     \multirow{7}{*}{$40$} & $3$ & $19$ & $9.4\cdot 10^{-14}$ & $5.9\cdot 10^{-12}$ & $7.6\cdot 10^{-11}$ & $2.4\cdot 10^{-11}$ \\
     \cmidrule{4-6}
     & $4$ & $29$ & $4.1\cdot 10^{-13}$ & $5.4\cdot 10^{-11}$ & $6.9\cdot 10^{-10}$ & $1.4\cdot 10^{-10}$ \\
     \cmidrule{4-6}
     & $5$ & $171$ & $1.6\cdot 10^{-10}$ & $1.3\cdot 10^{-7}$ & $1.4\cdot 10^{-6}$ & $1.1\cdot 10^{-6}$ \\
     \cmidrule{4-6}
     & $6$ & $286$ & $4.5\cdot 10^{-9}$ & $5.9\cdot 10^{-6}$ & $1.1\cdot 10^{-4}$ & $6.0\cdot 10^{-5}$ \\
     \cmidrule{4-6}
     & $7$ & $969$ & $7.8\cdot 10^{-7}$ & $1.2\cdot 10^{-5}$ & $4.3\cdot 10^{-5}$ & $3.7\cdot 10^{-4}$ \\
     \bottomrule
    \end{tabular}}
\end{table}
}

\subsection{Incomplete tensor approximation}
We present numerical experiments for exploring the incomplete symmetric tensor approximation quality of Algorithm~\ref{algo:tensor-approx}. We first randomly generate the rank-$r$ symmetric $\cF$ as in \eqref{eq:randomF}. Then we generate a random tensor $\mathcal{E}$ with the same dimension and order as $\cF_m$ and scale it to a given norm $\epsilon$, i.e. $\|\mathcal{E}_m\|=\epsilon$. Let
$\widehat{\cF}_m=\cF_m+\mathcal{E}_m$. Algorithm~\ref{algo:tensor-approx} is applied to the subtensor $(\widehat{\cF}_m)_{\Omega}$ to compute the rank-$r$ approximation $\cF_m^*$. The approximation quality of $\cF_m^*$ can be measured by the absolute error and the relative error
\begin{equation*}
    \text{abs-err}_m\coloneqq\|(\cF_m^*-\cF_m)_{\Omega_m}\|,\quad\text{rel-err}_m\coloneqq\frac{\|(\cF_m^*-\widehat{\cF}_m)_{\Omega_m}\|}{\|(\mathcal{E}_m)_{\Omega_m}\|}.
\end{equation*}
We choose the values of $d,\,m,\,\epsilon$ as
\begin{equation*}
    d=15,\,25,\quad m=3,\,4,\,5,\,6,\quad\epsilon=0.1,\,0.01,\,0.001,
\end{equation*}
and $r$ as largest computable rank in Theorem~\ref{largest-r} given $d$ and $m$. For each $(d,\,m,\,r,\,\epsilon)$, we generate $100$ instances of $\widehat{\cF}_m$(for the case $(25,\,6,\,r,\epsilon)$, $20$ instances are generated due to long computational time) and record the minimum, average, maximum of $\text{abs-err}_m$, $\text{rel-err}_m$ respectively. For the case when $d=15$, the results are reported in Table~\ref{tensor-result1}. For the case when $d=25$, the results are reported in Table~\ref{tensor-result2}. For all instances, the output tensor of Algorithm~\ref{algo:tensor-approx} provides a good rank-$r$ approximation.

{\footnotesize
\begin{table}[htbp]
  \centering
  \caption{The performance of Algorithm~\ref{algo:tensor-approx} when $d=15$}
  \label{tensor-result1}
    \begin{tabular}{cccccccccc}
    \toprule
     &  & & &$\text{rel-error}$ & &  & &$\text{abs-error}$  \\
    \cmidrule{4-6} \cmidrule{8-10}
    $m$ & $r$ & $\eps$ & min & average & max & & min & average & max \\
    \midrule
    \multirow{4}{*}{3} & \multirow{4}{*}{6} & 0.1 & 0.8452 & 0.8953 & 0.9258 & & 0.0378 & 0.0444 & 0.0534 \\
     \cmidrule{4-6} \cmidrule{8-10}
     & & 0.01 & 0.8549 & 0.8947 & 0.9280 & & 0.0037 & 0.0045 & 0.0052 \\
     \cmidrule{4-6} \cmidrule{8-10}
     & & 0.001 & 0.8581 & 0.8969 & 0.9337 & & $3.5\cdot 10^{-4}$ & $4.4\cdot 10^{-4}$ & $5.1\cdot 10^{-4}$ \\
    \midrule
    \multirow{4}{*}{4} & \multirow{4}{*}{8} & 0.1 & 0.9382 & 0.9544 & 0.9666 & & 0.0256 & 0.0298 & 0.0346 \\
     \cmidrule{4-6} \cmidrule{8-10}
     &  & 0.01 & 0.9409 & 0.9569 & 0.9700 & & 0.0024 & 0.0029 & 0.0034 \\
     \cmidrule{4-6} \cmidrule{8-10}
     &  & 0.001 & 0.9333 & 0.9547 & 0.9692 & & $2.5\cdot 10^{-4}$ & $3.0\cdot 10^{-4}$ & $3.6\cdot 10^{-4}$ \\
    \midrule
    \multirow{4}{*}{5} & \multirow{4}{*}{15} & 0.1 & 0.9521 & 0.9612 & 0.9689 & & 0.0248 & 0.0275 & 0.0306 \\
     \cmidrule{4-6} \cmidrule{8-10}
     &  & 0.01 & 0.9529 & 0.9613 & 0.9690 & & 0.0025 & 0.0028 & 0.0030 \\
     \cmidrule{4-6} \cmidrule{8-10}
     &  & 0.001 & 0.9539 & 0.9615 & 0.9704 & & $2.4\cdot 10^{-4}$ & $2.7\cdot 10^{-4}$ & $3.0\cdot 10^{-4}$ \\
     \midrule
     \multirow{4}{*}{6} & \multirow{4}{*}{20} & 0.1 & 0.9625 & 0.9697 & 0.9767 & & 0.0215 & 0.0244 & 0.0271 \\
     \cmidrule{4-6} \cmidrule{8-10}
     &  & 0.01 & 0.9620 & 0.9694 & 0.9737 & & 0.0023 & 0.0025 & 0.0027 \\
     \cmidrule{4-6} \cmidrule{8-10}
     &  & 0.001 & 0.9619 & 0.9696 & 0.9769 & & $2.1\cdot 10^{-4}$ & $2.4\cdot 10^{-4}$ & $2.7\cdot 10^{-4}$ \\
     \bottomrule
    \end{tabular}
\end{table}
}

{\footnotesize
\begin{table}[htbp]
  \centering
  \caption{The performance of Algorithm~\ref{algo:tensor-approx} when $d=25$}
  \label{tensor-result2}
    \begin{tabular}{cccccccccc}
    \toprule
     &  & & &$\text{rel-error}$ & &  & &$\text{abs-error}$ \\
    \cmidrule{4-6} \cmidrule{8-10}
    $m$ & $r$ & $\eps$ & min & average & max & & min & average & max \\
    \midrule
    \multirow{4}{*}{3} & \multirow{4}{*}{11} & 0.1 & 0.9248 & 0.9377 & 0.9488 & & 0.0316 & 0.0347 & 0.0380 \\
     \cmidrule{4-6} \cmidrule{8-10}
     & & 0.01 & 0.9257 & 0.9380 & 0.9484 & & 0.0032 & 0.0035 & 0.0038 \\
     \cmidrule{4-6} \cmidrule{8-10}
     & & 0.001 & 0.9239 & 0.9383 & 0.9504 & & $3.1\cdot 10^{-4}$ & $3.4\cdot 10^{-4}$ & $3.8\cdot 10^{-4}$ \\
    \midrule
    \multirow{4}{*}{4} & \multirow{4}{*}{16} & 0.1 & 0.9809 & 0.9840 & 0.9861 & & 0.0166 & 0.0178 & 0.0194 \\
     \cmidrule{4-6} \cmidrule{8-10}
     &  & 0.01 & 0.9813 & 0.9839 & 0.9868 & & 0.0016 & 0.0018 & 0.0019 \\
     \cmidrule{4-6} \cmidrule{8-10}
     &  & 0.001 & 0.9808 & 0.9838 & 0.9860 & & $1.7\cdot 10^{-4}$ & $1.8\cdot 10^{-4}$ & $1.9\cdot 10^{-4}$ \\
    \midrule
    \multirow{4}{*}{5} & \multirow{4}{*}{55} & 0.1 & 0.9854 & 0.9870 & 0.9878 & & 0.0156 & 0.0161 & 0.0170 \\
     \cmidrule{4-6} \cmidrule{8-10}
     &  & 0.01 & 0.9858 & 0.9871 & 0.9884 & & 0.0015 & 0.0016 & 0.0017 \\
     \cmidrule{4-6} \cmidrule{8-10}
     &  & 0.001 & 0.9856 & 0.9870 & 0.9882 & & $1.5\cdot 10^{-4}$ & $1.6\cdot 10^{-4}$ & $1.7\cdot 10^{-4}$ \\
     \midrule
     \multirow{4}{*}{6} & \multirow{4}{*}{84} & 0.1 & 0.9938 & 0.9940 & 0.9943 & & 0.0106 & 0.0109 & 0.0111 \\
     \cmidrule{4-6} \cmidrule{8-10}
     & & 0.01 & 0.9939 & 0.9942 & 0.9946 & & 0.0011 & 0.0011 & 0.0011 \\
     \cmidrule{4-6} \cmidrule{8-10}
     & & 0.001 & 1.0001 & 1.0046 & 1.0102 & & $1.6\cdot 10^{-4}$ & $1.8\cdot 10^{-4}$ & $2.1\cdot 10^{-4}$ \\
     \bottomrule
    \end{tabular}
\end{table}
}

\subsection{Gaussian Mixture Approximation}
We explore the performance of Algorithm~\ref{algo:gaussian} for learning diagonal Gaussian mixture model. We compare it with the classical EM algorithm using the \texttt{MATLAB} function \texttt{fitgmdist} (\texttt{MaxIter} is set to be 100 and \texttt{RegularizationValue} is set to be $0.001$). The dimension $d=15$ and the orders of tensors $m=3,\,4,\,5,\,6$ are tested. The largest possible values of $r$ as in Theorem \ref{largest-r} are tested for each $(d, m)$. We generate $20$ random instances of $\{(\omega_i,\mu_i,\Sigma_i):i=1,\ldots,r\}$ for each $(d, m)$. For the weights $\omega_1,\ldots,\omega_r$, we randomly generate a positive vector $s\in \mathbb{R}^r$ and let $\omega_i=\frac{s_i}{\sum_{i=1}^r \omega_i}$. For each diagonal covariance matrix $\Sigma_i\in\mathbb{R}^{d\times d}$, we use the square of a random vector generated by \texttt{MATLAB} function \texttt{randn} to be diagonal entries. Each example is generated from one of $r$ component Gaussians and the probability that the sample comes from the $i$th Gaussian is the weight $\omega_i$. Algorithm~\ref{algo:gaussian} and EM algorithm are applied to learn the Gaussian mixture model from samples. After obtaining estimated parameters $(\omega_i,\,\mu_i,\,\Sigma_i)$ of the model, the likelihood of the sample for each component Gaussian distribution is calculated and we assign the sample to the group that corresponds to the maximum likelihood. We use classification accuracy, i.e. the ratio of correct assignments, to measure the performance of two algorithms. The accuracy comparison between two algorithms is shown in Table~\ref{gaussian-result}. As one can see, the performance of Algorithm~\ref{algo:gaussian} is better than EM algorithm in all tested cases.

{\footnotesize
\begin{table}[htbp]
  \centering
  \caption{Comparison between Algorithm \ref{algo:gaussian} and EM for learning Gaussian mixtures}\label{gaussian-result}
    \begin{tabular}{ccccc}
    \toprule
     & & &\multicolumn{2}{c}{accuracy}\\
    \cmidrule{4-5}
    $d$ & $m$ & $r$ & Algorithm \ref{algo:gaussian} & EM\\
    \midrule
    \multirow{5}{*}{15} & 3 & 6 & 0.9839 & 0.9567\\
    \cmidrule{4-5}
     & 4 & 8 & 0.9760 & 0.9451\\
     \cmidrule{4-5}
     & 5 & 15 & 0.9639 & 0.9382\\
     \cmidrule{4-5}
     & 6 & 20 & 0.9423 & 0.9285\\
     \bottomrule
    \end{tabular}
\end{table}
}

\section{Conclusion}
In this paper, we propose a novel approach to learn diagonal Gaussian mixture models using tensor decompositions. We first formulate the problem as an incomplete tensor decomposition problem and develop Algorithm \ref{algo:iSTD} to recover the low-rank decomposition given partial entries in the tensor. For a given dimension $d$ and order $m$, we find the largest rank our algorithm can compute in Theorem \ref{largest-r}. Then, we extend the tensor decomposition algorithm to find low-rank tensor approximations when errors exist in the given subtensor. It is summarized in Algorithm \ref{algo:tensor-approx}. Theorem \ref{thm:tensor} proves that the low-rank approximation generated by Algorithm \ref{algo:tensor-approx} is highly accurate if the input subtensor is close enough to the exact subtensor. Finally, we apply the incomplete tensor approximation algorithm to recover unknown parameters in Gaussian mixture models from samples. The unknown parameters are obtained by solving linear equations. It is presented in Algorithm \ref{algo:gaussian}. We prove in Theorem \ref{thm:gaussian approx} that the recovered parameters are close to the true parameters when the estimated high-order moments are accurate.

The proposed algorithms for learning diagonal Gaussian mixtures can handle a large number of Gaussian components by utilizing higher order moments. For general Gaussian mixture models with non-diagonal covariances, how to learn their unknown parameters using symmetric tensor decompositions and generating polynomials is an interesting question to investigate in the future.


\end{document}